\documentclass[11pt,notitlepage,twoside]{article}
\usepackage{latexsym}
\usepackage{amsmath}
\usepackage{amsfonts}
\usepackage{amssymb}
\usepackage{amscd}
\renewcommand{\a }{\alpha }
\renewcommand{\d}{\delta }
\newcommand{\D }{\Delta }

\newcommand{\e }{\varepsilon }
\newcommand{\g }{\gamma}

\newcommand{\G }{\Gamma }
\renewcommand{\l }{\lambda }

\newcommand{\n }{\nabla }

\newcommand{\s }{\sigma }

\renewcommand{\th }{\theta }

\renewcommand{\O }{\Omega }
\newcommand{\z }{\zeta }
\newcommand{\ov}{\overline}
\newcommand{\intbar}{\mathop{\int\makebox(-13.5,0){\rule[4pt]{.7em}{0.3pt}}%
\kern-6pt}\nolimits}

\newcommand{\be}{\begin{equation}}
\newcommand{\ee}{\end{equation}}
\newcommand{\bes}{\begin{equation*}}
\newcommand{\ees}{\end{equation*}}
\newcommand{\ba}{\begin{eqnarray}}
\newcommand{\ea}{\end{eqnarray}}
\newcommand{\bas}{\begin{eqnarray*}}
\newcommand{\eas}{\end{eqnarray*}}
	\newenvironment{proof}[1][Proof]{\noindent\textbf{#1.} }{\ \rule{0.5em}{0.5em}}
\newenvironment{pf}{\noindent{\sc Proof}.\enspace}{\rule{2mm}{2mm}\medskip}
\newenvironment{pfn}{\noindent{\sc \bf Proof }}{\rule{2mm}{2mm}\medskip}

\newcommand{\R}{\mathbb{R}}

\newcommand{\N}{\mathbb{N}}

\author{ Mohamed Ben Ayed$^a$\thanks{  E-mail : {M.BenAyed@qu.edu.sa} } and
 Khalil El Mehdi$^{a,b}$\thanks{Corresponding author. E-mail : {K.Jiyid@qu.edu.sa} } \\ 
{\footnotesize
a : Department of Mathematics, College of Science, Qassim University, Buraydah 51542, Saudi Arabia}\\
{\footnotesize
 b : Facult\'e des Sciences et Techniques, Universit\'e de Nouakchott, Nouakchott 2373, Mauritania.}\qquad\quad
}
\date{}

\title{\bf Bubbles clustered inside for almost critical problems} 
\begin{document}

\newtheorem{lem}{Lemma}[section]
\newtheorem{pro}[lem]{Proposition}
\newtheorem{thm}[lem]{Theorem}
\newtheorem{rem}[lem]{Remark}
\newtheorem{cor}[lem]{Corollary}
\newtheorem{df}[lem]{Definition}

\maketitle

\noindent{\bf Abstract:} We investigate the existence of blowing-up solutions of  the following almost critical problem  
$$
-\Delta u +V(x)u =u^{p-\e},\quad u>0\quad\mbox{in}\quad  \O,\quad u=0\quad\mbox{on}\quad \partial\O,
$$
 where $\O$ is a bounded regular domain  in $\mathbb{R}^n$, $n\geq 4$, $\varepsilon$ is a small positive parameter, $p+1=(2n)/(n-2)$ is the critical Soblolev exponent and    the potential $V$ is a smooth positive function. 
We find solutions which exhibit bubbles clustered inside as $\e$ goes to zero. To the best of our knowledge, this is the first existence result for interior non-simple blowing-up positive solutions to Dirichlet problems in general domains.
Our results are proven through delicate asymptotic
estimates of the gradient of the associated Euler-Lagrange functional.

\bigskip

\noindent{\bf Key Words:}  Partial Differential Equations, Variational analysis, Nonlinear analysis, Critical Sobolev exponent.

\bigskip

\noindent {\bf MSC  $2020$}: 35A15, 35J20,  35J25.

\section{Introduction and Main Results}
In this paper, we study the following almost critical problem
\begin{align}
(\mathcal{P}_{V,\e}):\qquad 
\begin{cases}
-\Delta u +V u = u^{p-\varepsilon}\quad &\mbox{in}\quad \Omega,
\\
\quad  u > 0 \quad &\mbox{in}\quad \Omega,
\\
\quad u =0\,\quad &\mbox{on}\quad \partial\Omega,
\end{cases}
\end{align}
where $\Omega$ is a bounded regular domain in $\mathbb{R}^n$, $n\geq 4$, $p+1= (2n)/(n-2)$ is the critical Sobolev exponent for the embedding $H^1_0(\Omega) \hookrightarrow L^q(\Omega)$, the potential $V$ is a $C^3$ positive function on $\overline{\Omega}$ and $\varepsilon$ is a small positive parameter.  

The problem in the form of $(\mathcal{P}_{V,\e})$ arises in various physical models, such as quantum transport and non-relativistic Newtonian gravity, as detailed in \cite{BGDB, BLJS, Pe} and their associated references. It is also connected to the Yamabe problem in differential geometry, as discussed for example in \cite{DH} and the references therein.

In the critical case, where $\e=0$, it is well-established that the existence of solutions to problem $(\mathcal{P}_{V,\e})$ is influenced by the geometry of the domain, the characteristics of the potential $V$, and the dimension $n$. For instance, when $V$ is constant and the domain is star-shaped, there are no solutions to the given problem. Due to the vast amount of research on this subject, we will only mention the seminal works by Brezis-Nirenberg \cite{BN} and Bahri-Coron \cite{BC}.

In the subcritical case where $\e >0$, proving the existence of a solution to the problem $(\mathcal{P}_{V,\e})$ is relatively simple. This can be shown by observing that the infimum 
$$
\inf \,\{\int_\O \left(|\n u |^2 +V\,u^2\right):\quad u\in H^1_0(\O) \,  \mbox{ and } \,  \int_\O |u|^{p+1-\e}=1\, \}
$$
is achieved, thanks to the compactness of the embedding $H^1_0(\Omega) \hookrightarrow L^{p+1-\e}(\Omega)$.

In \cite{2ABE}, solutions of $(\mathcal{P}_{V,\e})$ were constructed that concentrate, as $\e \to 0$, at interior blow-up points, forming isolated bubbles. By {\it bubble} here, we refer to the functions defined by
$$
\delta_{a,\lambda}(x):= c_{0}\frac{{\lambda}^{ (n-2) / 2 }}{\left ( 1+\lambda^{2} |x-a|^{2}\right )^{ (n-2) /  2 }} , \quad \mbox{with} \quad c_{0}:=\left [ n(n-2) \right ]^{ \frac{n-2}{ 4}}, \quad a\in\mathbb{R}^n,\quad \l>0
$$
which are the only solutions to the equation \cite{CGS}
$$
-\D  u = u^p,\quad u>0 \quad \mbox{in}\quad  {\mathbb{R}}^{n}.
$$

   In this paper, we focus on the construction of interior bubbling solutions of $(\mathcal{P}_{V,\e})$ with clustered bubbles at  critical points of $V$. These solutions reveal a new phenomenon for positive solutions, namely the existence of non simple blow-up points in the interior for the subcritical problem with Dirichlet boundary conditions in general domains. This phenomenon has been observed for changing-sign solutions on certain symmetric domains \cite{PW} and for positive solutions on the ball \cite{WY}. To the best of our knowledge, this is the first existence result for interior non-simple blowing-up positive solutions to Dirichlet problems in general domains.

In order to formulate our results, we need to introduce some notation. Let $ b $ be a non-degenerate critical point of $ V $ and $N\in\mathbb{N}$, we define the following function 
\be \label{FYN}
\mathcal{F}_{b, N } (z_1, \cdots , z_N) := \sum_{ i=1}^N D^2 V(b) ( z_i , z_i ) - \sum_{ 1 \leq j \neq i \leq N  } \frac{ 1 } { | z_j - z_i |^{ n-2} }. 
\ee

The aim of our first result is to construct interior bubbling solutions with clustered bubbles at a critical point of $V$, Namely, we have:
\begin{thm}\label{th:t1}
Let   $n \geq 4$, $ N  \geq 2$ and $ b \in \O$ be a non-degenerate critical point of $V$. Assume that the function $  \mathcal{F}_{ b , N }$ has a non-degenerate critical point $ ( \ov{z}_1, \cdots, \ov{z}_N) $. 
Then, there exists a small positive real $\e_0$  such that for any  $\e\in (0,\e_0]$,  problem $(\mathcal{P}_{V,\e})$ has a solution $u_{\e, b}$ with the following property
$$
u_{\e,b} =\sum_{i=1}^N \a_{i,\e}\d_{{a_{i,\e}}, \l_{i,\e}} +v_\e,
$$
with  
\begin{align}
&\| v_\e\| < c \sqrt{\e},\label{t30}\\
&|\a_{i,\e} -1| \leq \e \ln^2 \e \quad \forall\,1\leq i\leq N\label{t31}\\
&\frac{1}{c} \leq \frac{\ln^{\sigma_n}\l_{i,\e}}{\e\l_{i,\e}^2}\leq c  \quad \forall\,1\leq i\leq N\label{t32}\\
&|a_{i,\e} -b +\eta(\e)\sigma\overline{z}_i|\leq \eta(\e) \eta_0 \quad \forall\,1\leq i\leq N,\label{t33}
\end{align}
where $\eta_0$ is a small positive constant, $\sigma_4=1$, $\sigma_n=0$ for $n\geq 5$, $c$ is a positive constant, $\eta(\e)$ and $\sigma$ are defined in \eqref{sigma}.\\
In addition, if for each $N$,  $\mathcal{F}_{b, N }$ has a non-degenerate critical point, then, problem $(\mathcal{P}_{V,\varepsilon})$  has an arbitrary number of non-constant distinct solutions provided that $\e$ is small.
\end{thm}

\begin{rem} 
Note that, for $ N=1 $, the existence of solution blowing up at a critical point of $V$ with only one bubble has already proved in \cite{2ABE}.
\end{rem}

Our second result addresses the case of multiple interior blow-up points with clustered bubbles. More precisely, we prove:

\begin{thm}\label{th:t2} 
Let  $n \geq 4$,  $ b_1, \cdots , b_k  \in \O$ be  non-degenerate critical points of $ V $ and 
 $ N_1, \cdots, N_k$ be  non-zero natural integers. For each  $ i \in \{ 1 , \cdots, k \}$ such that $ N_i \geq 2$, assume that the function $  \mathcal{F}_{ b_i , N_i }$ has a non-degenerate critical point $( \ov{z} ^ i _1, \cdots, \ov{z}^ i _{N_i}) $. Then, there exists a small positive real $\e_0$ such that for any  $\e\in (0,\e_0]$,  problem $(\mathcal{P}_{V,\e})$ has a solution  $u_{\e, b_1,\cdots, b_k}$ satisfying 
$$
u_{\e, b_1,\cdots, b_k} =\sum_{i=1}^{N_1} \a_{1,i,\e}\d_{{a_{1,i,\e}}, \l_{1,i,\e}}+\cdots + \sum_{i=1}^{N_k} \a_{k,i,\e}\d_{{a_{k,i,\e}}, \l_{k,i,\e}} +v_\e,
$$
with $\| v_\e\| \leq c \sqrt{\e}$  and for each $l\leq k$ the coefficients $\a_{l,i,\e}$, the speeds $\l_{l,i,\e}$ and the points $a_{l,i,\e}$ satisfy the properties \eqref{t31}, \eqref{t32}, and \eqref{t33}, respectively. 
\end{thm}

The proof of our results relies on refined asymptotic estimates of the gradient of the associated Euler-Lagrange functional in the  neighborhood of bubbles. The goal is to determine the equilibrium conditions satisfied by the concentration parameters. These balancing conditions are derived by testing the equation with vector fields that represent the dominant terms of the gradient relative to the concentration parameters. Analyzing these conditions provides all the necessary information to establish our results.

The rest of the paper is organized as follows: In Section $2$, we introduce the parameterization of the neighborhood of bubbles, provide a precise estimate of the infinite-dimensional part, and carry out a delicate asymptotic expansion of the gradient of the associated Euler-Lagrange functional. Section $3$ is dedicated to proving our results. Finally, in Section $4$, we present several estimates that are referenced throughout the paper.

\section{Analytical Framework}
Problem $(\mathcal{P}_{V,\e})$ has a variational structure and its Euler-Lagrange functional is defined on $H^1_0(\O)$ by 
\be\label{eq:7}
 I_{V,\varepsilon}\left(u \right) :=\frac{1}{2} \int_{\Omega}\left|\nabla u \right|^{2} +\frac{1}{2} \int_{\Omega} Vu^2-\frac{1}{p+1-\varepsilon}\int_{\Omega}\left|u \right|^{p+1-\varepsilon} \quad 
(\mbox{with } p := \frac{n+2}{n-2}) . \ee
In the sequel, we will use  the following scalar product and its corresponding norm defined by
\be\label{eq:4}
\left\langle v,w\right\rangle:=\int_{\O}\n v\cdot\n w+\int_{\O}Vvw \quad ; \quad \|w\| ^2 :=\int_{\O}|\n w|^2+\int_{\O}V w^2.
\ee
Notice that, since $V$ is a $C^0$-positive function on $\ov{\O}$, we deduce that this norm is equivalent to the two norms $||.||_{0}$ and $||.||_{1}$ of $H^1_0(\O)$ and $H^1(\O)$ respectively.

We start by giving some estimates concerning the approximate solutions for $(\mathcal{P}_{V,\e})$. For $a\in\O$ and $\l>0$, let   $ \pi \d_{a, \l} $ the projection defined by 
$$
\begin{cases}
-\Delta\pi \delta_{a,\lambda} +V \pi \delta_{a,\lambda} = \delta_{a,\lambda}^{\frac{n+2}{n-2}}\quad &\mbox{in}\quad \Omega,  \\
\quad\pi \delta_{a,\lambda} =0\,\quad &\mbox{on}\quad \partial\Omega , 
\end{cases}
$$
and  
$$
\th_{a,\l}:=\d_{a,\l}-\pi\d_{a,\l}.
$$
These functions are introduced in \cite{2ABE} and they satisfy, for $ n \geq 4$,  
\begin{align}
 & 0\leq\pi\d_{a,\l}\leq\d_{a,\l} \quad ; \quad \left|\l\frac{\partial\pi\d_{a,\l}}{\partial\l}\right|\leq\frac{n+2}{2}\pi\d_{a,\l} \quad ; \quad \left|\frac{1}{\l}\frac{\partial\pi\d_{a,\l}}{\partial a_j}\right|\leq\frac{n+2}{2}\pi\d_{a,\l}   \quad  \mbox{ in } \O , \label{est0} \\
 &  0\leq\th_{a,\l}\leq\d_{a,\l} \quad  ; \quad  \left|\l\frac{\partial\th_{a,\l}}{\partial\l}\right|\leq\frac{n-2}{2}\theta_{a,\l} \quad  ; \quad \left|\frac{1}{\l}\frac{\partial\theta_{a,\l}}{\partial a_j}\right|\leq\frac{n-2}{2}\th_{a,\l}  \quad  \mbox{ in } \O ,\label{est1} \\
&   \th_{a,\l} \leq cR_{a,\l}^1\d_{a,\l} \quad ; \quad   \left|\frac{1}{\l}\frac{\partial\th_{a,\l}}{\partial a_j}\right|\leq cR_{a,\l}^2\d_{a,\l}\quad
\mbox{ in } \O_0:= \{ x \in \O : d(x, \partial \O) \geq d_0 \} , \label{est2}  \\
& \mbox{ with }
 R_{a,\l}^ 1(x):= { \ln ^{\s_n}(\l) }{ \l^{ -2 } } + |x-a|^{2} | \ln | x - a | |^{\s_n} \quad   ; \quad  
R_{a,\l}^2(x) := \l^{-2}+\l^{-1}|x-a| ,  \label{est3} 
\end{align}
 where $ d_0$ is any fixed small positive constant,  $ \s_4 := 1$,  $ \s_n := 0 $ for $ n \geq 5$ and $a_j$ is the $j^{th}$-component of $a$. \\
 For the proof of these facts, the interested reader is referred to \cite{2ABE}.

Now, we introduce the parameterization of the neighborhood of bubbles. Let $ N \in \N$ and $ \mu > 0 $ be a small real. We define the following set 
\begin{align}
 \mathcal{O} ( N, \mu) := \{ (\a, a, \l) \in (0,\infty)^N \times  \O ^N \times  ( \mu^{-1} , \infty)^N & : \, \,  | \a_i - 1 | < \mu ; \, \, d(a_i, \partial \O ) > 2 d_0 ;  \label{eq:2} \\
  &  \e \ln \l_i < \mu ; \, \, \forall \, \, i  \, \mbox{ and }\, \,  \e_{ij} < \mu  \, \, \forall \, \, i \neq j \}  \nonumber 
  \end{align}
where  $ \e_{ij} $ is defined by 
\be \label{eq:3} \e_{ij} := \Big( \frac{ \l_i }{ \l_j } + \frac{ \l_j }{ \l_i } + \l_i \l_j | a_i - a_j |^2 \Big)^{(2-n)/2} . \ee
Observe that, since $ \O $ is bounded, for each $ b \in \O $ and $ \lambda > 0 $ satisfying $ \e \ln \l $  is small, it follows that 
  \be\label{eq:8} 
  \d_{b, \l} ^{ - \e } = c_0^{- \e} \l^{-\e (n-2)/2} \Big( 1 + \frac{ n-2 }{2 } \e \ln(1+ \l^2 | x-b|^2 ) \Big) + O \Big( \e^2 \ln^2 (1+ \l^2 | x-b|^2 ) \Big) = 1 + o(1)  . \ee
For each $( \a, a , \l) \in \mathcal{O}(N,\mu) $ and $ v \in E_{a,\l} ^\perp $, we associate a function 
\begin{align} 
& u := \sum _{ i=1}^N \a_i (\pi \d_{a_i, \l_i} ) + v := \underline{ u } + v , \quad \mbox{  where } \label{eq:6}  \\ 
& E_{a,\l} := \mbox{span} \Big\{ \pi \d_{a_i, \l_i} , \frac{ \partial (\pi \d_{a_i, \l_i}) }{ \partial \l_i } ,  \frac{ \partial (\pi \d_{a_i, \l_i}) }{ \partial a_{i,j}} , \, \, i \in \{1, \cdots , N \} , \, \, j \in \{1, \cdots , n \} \Big\}.  \label{Eal} \end{align}
We notice that the orthogonality is taken with respect to the scalar product defined in \eqref{eq:4}. \\
In the sequel, we need to study the functional $ I_{V,\e} $ and to find some positive critical points $ u $ having the form \eqref{eq:6}. We start by studying the $v$-part of $u$.

\subsection{Estimate of the infinite-dimensional part}
Let $(\alpha, a,\lambda) \in \mathcal{O} (N,\mu)$, $ v \in E_{a,\l}^\perp$  and $u$ be defined in \eqref{eq:6}. In this section, we are going to study  the $v$-part of $u$. To this aim, we need to expand $I_{V, \varepsilon}$ with respect to $v$.

Observe that, for $ b_1 $, $ b_2 \in \R $, and $\g>2$, it holds 
\be \label{lst1} 
	\left| \left| b_1 + b_2 \right|^{\g}-\left| b_1 \right|^{\g}-{\g}\left| b_1 \right|^{{\g}-2} b_1 b_2 - \frac{1}{2} {\g}(\g -1) \left| b_1 \right|^{\g -2} b_2 ^{2}\right|\leq \begin{cases}
		c\left| b_1 \right|^{\g -3} \left| b_2 \right|^3 + c | b_2 | ^\gamma  & \mbox{if} \quad \g > 3, \\
		c \left| b_2 \right|^{\g} & \mbox{if} \quad \g \leq 3.
\end{cases}
\ee
Thus, using  the fact that $\left\langle v,\underline{u}\right\rangle=0$, we get: 
\begin{align} 
&   I_{V, \varepsilon}\left(u\right)  = I_{V,\varepsilon}(\underline{u})-  f_{\varepsilon}(v)+\frac{1}{2}Q_{\varepsilon}(v)
+o \left({||v||}^{2}\right)  \qquad \mbox{ where } \label{ie1}  \\
 &   f_{\varepsilon}(v) := \int_{\Omega}(\underline{u})^{p-\varepsilon}v  \qquad \mbox{ and } \qquad Q_{\varepsilon}(v):=\left||v|\right|^{2} -(p-\varepsilon) \int_{\Omega}({\underline{ u }})^{p-\varepsilon-1}  v^{2}.\label{1}
 \end{align}
 Notice that $ Q_\e $ is a positive definite quadratic form (see Proposition $3.1$ of \cite{2ABE}) and the linear form $ f_\e $ satisfies 
\begin{lem} \label{lfv}
	Let $(\alpha, a,\lambda) \in \mathcal{O} (N,\mu)$. Then, for each $v \in E_{a,\lambda}^{ \perp}$, we have:
\begin{align*}  
& \left| f_{\varepsilon} (v) \right| \leq c \left\| v \right\| \Big(\varepsilon+ \sum_{ 1 \leq i \leq N} T_{2}(\lambda_{i}) +  \sum_{i\neq j} T_3( \e_{ij} ) \Big) , 
\qquad \mbox{ 
where }  \\
&  T_2(\l):= \big( \lambda^{-2} \, \, \,   \mbox{ if } \;\, n\geq 5\quad ; \quad   \lambda^{-2}\ln\l  \, \, \,   \mbox{ if } \;\, n=4 \quad ; \quad  \lambda^{-1} \, \, \,   \mbox{ if } \;\, n=3 \big) ,  \\
 &  T_3 (t) := ( t ^{ (n+2) / (2(n-2)) } ( \ln t ^{-1} )^{ (n+2) / 2n }  \text{ if } \; n\geq 6 \quad , \quad 
t  \text{ if }\;  n\leq 5 ) . 
\end{align*}
\end{lem}

\begin{proof} First, observe that, for  $\beta_{i}>0$ and $\g>1$, it holds:
\be \label{lst2} 
\left| \left(\sum_{i=1}^{N} \beta_{i}\right)^{\g} - \sum_{i=1}^{N} \beta_{i}^{\g} \right|\leq  \begin{cases}
	c \sum_{i \neq j} \left( \beta_{i} \beta_{j} \right)^{ {\g} / {2}} & \mbox{if} \quad \g \leq 2, \\
	c \sum_{i \neq j} \beta_{i}^{\g -1} \beta_{j} & \mbox{if} \quad \g > 2.
\end{cases}
\ee 
Thus, we obtain 
\be \label{eq:f3}
 f_{\varepsilon}(v) = \sum_{i=1}^{N}\alpha_{i}^{p-\varepsilon} \int_{\Omega} \left( \pi \delta_{a_{i},\lambda_{i}} \right)^{p-\varepsilon}v + \sum_{i \neq j} \begin{cases}
			O \left( \int_{\Omega} \left[ (\pi \delta_{a_{i},\lambda_{i}}) ( \pi \delta_{a_{j},\lambda_{j}})\right]^{\frac{p-\varepsilon }{2}} \left| v \right|\right)& \text{ if } \; n\geq 6, \\ 
			O \left( \int_{\Omega}  (\pi \delta_{a_{i},\lambda_{i}})^{p-\varepsilon -1}( \pi \delta_{a_{j},\lambda_{j}}) \left| v \right| \right)& \text{ if }\;  n\leq 5.
		\end{cases}
	\ee 
	Observe that, using \eqref{est0}  and \eqref{eq:8}, we get 
	\begin{align}
		\int_{\Omega} \left[\left( \pi \delta_{a_{i},\lambda_{i}}\right) \left( \pi \delta_{a_{j},\lambda_{j}}\right)\right]^{\frac{p-\varepsilon}{2}}|v| &\notag \leq \int_{\Omega} \left(  \delta_{a_{i},\lambda_{i}}\delta_{a_{j},\lambda_{j}}\right)^{\frac{p}{2}}|v|  \\
		& \leq \left [ \int_{\Omega} \left( \delta_{a_{i},\lambda_{i}}\delta_{a_{j},\lambda_{j}}\right)^{\frac{n}{n-2}} \right ]^{ (n+2) / 2n } \left( \int _{\Omega } \left| v\right|^{\frac{2n}{n-2}}\right) ^ { (n-2) / (2n) } .
	\end{align}
	Note that, Estimate (E2) of \cite{Bahri-book} gives us 
	\begin{align} \label{eq:f1}
		\int_{\mathbb{R}^{n}} \left( \delta_{a_{i},\lambda_{i}}\delta_{a_{j},\lambda_{j}}\right)^{ n / (n-2) } \leq c \varepsilon_{ij}^{\frac{n}{n-2} }\ln(\varepsilon_{ij}^{-1})
	\end{align}
	and using the embedding theorem of $H_{0}^{1}(\Omega) \hookrightarrow L^{\frac{2n}{n-2}}(\Omega)$, we derive that
	\begin{equation}\label{eq:f2}
		\int_{\O}\left[(\pi \d_{a_i,\l_i})(\pi \d_{a_j, \l_j})\right]^{\frac{p-\e}{2}}|v|\leq c \e_{ij}^{\frac{n+2}{2(n-2)}}\left(\ln\e_{ij}^{-1}\right)^{\frac{n+2}{2n}}\|v\|.
	\end{equation}
	In the same way, for $n\leq 5$ (which implies $p-1:=\frac{4}{n-2} >1$), using Lemma 2.2 of \cite{Dammak}, we have 
	\begin{align}
 \int_{\O}(\pi \d_{a_i, \l_i})^{p-\e-1}(\pi \d_{a_j, \l_j})|v| & \leq \int_{\O}\d_{a_i, \l_i}^{p-1}\d_{a_j, \l_j}|v|  \leq  c\|v\|\left(\int_{\O} \d_{a_i,\l_i} ^{(p-1) \frac{2n }{n+2} }\d_{a_j, \l_j}^{\frac{2n}{n+2}}\right)^{ \frac{n+2}{2n} }  \leq c\|v\|\e_{ij} . \label{eq:f4}
	\end{align}
	It remains to estimate the first integral in \eqref{eq:f3}. Observe that, for  $ 0 < b_2 < b_1 $ and $\g  > 1 $,  it holds:
\be \label{lst3} 
	\left( b_1 - b_2 \right)^{\g} = b_1 ^{\g} + O \left( b_1 ^{\g -1} b_2 \right) 
\ee 
thus,  we get
	$$\int_{\O}\left(\pi \d_{a_i, \l_i}\right)^{p-\e}v=\int_{\O}\left(\d_{a_i, \l_i}-\th_{a_i,\l_i}\right)^{p-\e}v=\int_{\O}\d_{a_i, \l_i}^{p-\e}v+O\left(\int_{\O}\d_{a_i, \l_i}^{\frac{4}{n-2}-\e}\th_{a_i,\l_i}|v|\right).$$
The last integral can be deduced following the proof of  Lemma 6.5 of  \cite{2ABE} by using \eqref{eq:8} and we get 
\be \label{R1dv} 
 \int_{\O}\d_{a,\l}^{ \frac{4}{n-2} - \e  }\th_{a,\l} |v| \leq c \int_{\O}\d_{a,\l}^{ 4/(n-2) }\th_{a,\l}|v|\leq c\|v\|T_2(\l) . 
\ee
For the other one, using \eqref{eq:8}, and the fact that $v \perp \pi\d_{a_i, \l_i}$, we get:
	\begin{align}\label{eq:f5}
\int_{\O}\d_{a_i, \l_i}^{p-\e} v = & c_{0}^{- \e}\l _i ^{-\e\frac{n-2}{2}}\int_{\O}\d_{a_i,\l_i}^{p}v+O\left(\e\int_{\O}\d_{a_i, \l_i}^{p}\ln(1+\l^{2}_i|x-a_{i}|^{2})|v|\right)\notag\\
		= &	O\left(\e\|v\|\left(\int_{\O}\d_{a_i, \l_i}^{\frac{2n}{n-2}}\ln^{\frac{2n}{n+2}}(1+\l_{i} ^2 |x-a_{i}|^{2})\right)^{\frac{n+2}{2n}}\right)	= O\left(\e \|v\|\right).
	\end{align}
	Thus,  \eqref{R1dv} and \eqref{eq:f5} imply that 
	\begin{equation}\label{eq:f6}
		\int_{\O}\left(\pi \d_{a_i, \l_i}\right)^{p-\e}v=O \left(\e\|v\|+\|v\|T_{2}(\l_{i})\right) . 
	\end{equation}
	Putting \eqref{eq:f6}, \eqref{eq:f4} and \eqref{eq:f2} in \eqref{eq:f3}, the result follows.
\end{proof}

Combining  \eqref{ie1}, Lemma \ref{lfv} and the fact that $ Q_\e $ is a positive definite quadratic form, we deduce that
 \begin{pro}\label{v bar} 
 Let $(\alpha,a,\lambda)\in \mathcal{O} (N,\mu)$ and $\underline{u} := \sum_{i=1}^{N}\alpha_{i}\pi \delta_{a_{i},\lambda_{i}}$. Then for $\varepsilon$ small, there exists a unique $\ov{v}\in E_{a,\lambda }^{\perp }$ satisfying:
 \begin{align}
  \langle \n I_{V, \varepsilon } (\underline{u}+\ov{v}) , h \rangle  = 0 \quad \forall h \in E_{a,\lambda }^{\perp } .  \label{vbar}
 \end{align}
 Moreover, $\ov{v}$ satisfies 
 $$ \left\|{\ov{v}}\right\| \leq   c R_{v}(a,\lambda ) \quad  \mbox{ with }\quad  R_{v}(a,\lambda ) :=  \varepsilon + \sum_{i\neq j} T_3 ( \varepsilon_{ij} )  +  \sum _ { 1 \leq i \leq N } T_2 (\l_i) $$
 where $ T_2 $ and $ T_3 $ are defined in Lemma \ref{lfv}.
 \end{pro}
 
\subsection{ Expansion of the gradient of the associated functional}

In this section, we will provide asymptotic expansions of the gradient of functional $I_{V,\varepsilon}$. We begin with the expansion with respect to the gluing parameter $\a_i's$,  specifically proving:
\begin{pro}\label{gradalpha}
	Let $(\alpha, a,\lambda) \in \mathcal{O} (N,\mu)$, $v \in E_{a,\lambda}^{ \perp}$ and $u= \sum_{i=1}^{N}\alpha_{i}\pi\delta_{a_{i},\lambda_{i}}+v$. Then, for $\varepsilon$ small and $i\in \left\{1, \cdots ,N \right\}$, we have:
	$$
	\left< \nabla I_{V,\varepsilon}(u),\pi \delta_{a_{i},\lambda_{i}}\right> = \alpha_{i}S_{n} \left( 1-\alpha_{i}^{p-1-\varepsilon}\lambda_{i}^{ -\varepsilon (n-2) / 2 }\right)+O(R_{\alpha_{i}})
	$$
	where 
	$$  S_n:= \int_{\R^n } \d_{0,1} ^{ 2n / (n-2) }, \quad   \quad R_{\alpha_{i}}:=\varepsilon+T_{2}(\lambda_{i})+\left\| v\right\|^{2}+\sum_{i\neq j}\varepsilon_{ij}
	$$
and $ T_2$ is defined in Lemma \ref{lfv}.
\end{pro}

\begin{proof}
	Observe that, since $v \in E_{a,\lambda}^{ \perp}$, it holds:
	\begin{align}\label{eq:e1}
		\left< \nabla I_{V,\varepsilon},\pi \delta_{a_{i},\lambda_{i}}\right>= \sum_{j=1}^{N}\alpha_{j}\left< \pi \delta_{a_{j},\lambda_{j}},\pi \delta_{a_{i},\lambda_{i}}\right>-\int_{\Omega}\left|u\right|^{p-1-\varepsilon}u\pi \delta_{a_ {i},\lambda_{i}}.
	\end{align}
	First, for $ j \neq i $, we have:
	\begin{align} 
		\left< \pi \delta_{a_{j},\lambda_{j}},\pi \delta_{a_{i},\lambda_{i}}\right> & = \int_{\Omega}\n \left( \pi \delta_{a_{j},\lambda_{j}}\right)\n \left( \pi \delta_{a_{i},\lambda_{i}}\right)+\int_{\Omega}V \left( \pi \delta_{a_{j},\lambda_{j}}\right) \left( \pi \delta_{a_{i},\lambda_{i}}\right)  \nonumber \\
	& = \int_{\Omega} \left( -\D +V\right)\left( \pi \delta_{a_{j},\lambda_{j}}\right) \left( \pi \delta_{a_{i},\lambda_{i}}\right) \leq \int_{ \O }\delta_{a_{j},\lambda_{j}}^{\frac{n+2}{n-2} } \delta_{a_{i},\lambda_{i}}  \leq c \varepsilon_{ij}. \label{dij} 
	\end{align}
	by using Estimate (E1) of \cite{Bahri-book}.
	Second, for $j=i$, as in the previous computation, we have 
$$
\left\| \pi \delta_{a_{i},\lambda_{i}} \right\|^{2} = \int_{\Omega} \delta_{a_{i},\lambda_{i}}^{\frac{n+2}{n-2}}\left(\pi \delta_{a_{i},\lambda_{i}}\right)= \int_{\Omega} \delta_{a_{i},\lambda_{i}}^{\frac{2n}{n-2}}-\int_{\Omega} \delta_{a_{i},\lambda_{i}}^{\frac{n+2}{n-2}}\theta_{a_{i},\lambda_{i}}. 
$$ 
	The estimate of the first integral is well known and we have 
$$
\int_{\Omega} \delta_{a_{i},\lambda_{i}}^{\frac{2n}{n-2}} = S_{n}+O \left( \frac{1}{\lambda_{i}^{n}} \right).
$$
	For the second integral, let $B_{i}:= B \left( a_{i}, d_0 \right) $, using  \eqref{est1} and  \eqref{est2}, easy computations imply that 
	\begin{align}\label{eq:e3}
		\int_{\Omega} \delta_{a_{i},\lambda_{i}}^{\frac{n+2}{n-2}}\theta_{a_{i},\lambda_{i}} & \leq c \int_{B_{i}}R_{1}(x,a_{i},\lambda_{i}) \delta_{a_{i},\lambda_{i}}^{\frac{2n}{n-2}}+ \int_{\Omega \backslash B_{i}} \delta_{a_{i},\lambda_{i}}^{\frac{2n}{n-2}}  \leq c T_{2}(\lambda_{i}) + \frac{c}{\lambda_{i}^{n}}  \leq c T_{2}(\lambda_{i})  .
	\end{align}
Thus we deduce that 
\be \label{eqq:e1}	
\left\| \pi \delta_{a_{i},\lambda_{i}} \right\|^{2} = S_n + O  \big( \lambda^{-2} \, \, \,   \mbox{ if } \;\, n\geq 5\quad ; \quad   \lambda^{-2}\ln\l  \, \, \,   \mbox{ if } \;\, n=4 \quad ; \quad  \lambda^{-1} \, \, \,   \mbox{ if } \;\, n=3 \big) . \ee
	Now, we focus on the last term in \eqref{eq:e1}. Observe that, for 
 $ b_1 ,  b_2 ,  z \in \R $ such that $|z|\leq \beta | b_1 | $ for some positive constant $\beta$ and $\gamma > 0$,  it holds 
 \begin{align} 
 &  \left| b_1 + b_2 \right| ^{\gamma -1} \left (  b_1 + b_2 \right ) z = | b_1 | ^ {\gamma-1 } b_1 z + \gamma  | b_1 | ^ {\gamma -1} b_2 z + O\left( | b_1 | ^ {\gamma -1} b_2 ^{2} + \left| b_2 \right|^{\gamma +1}\right ) ,  \label{lst4} \\
 & \left| b_1 + b_2 \right| ^{\gamma} = \left| b_1 \right|^{\gamma} + O \left( | b_1 |^{\gamma -1} | b_2 | + | b_2 | ^{\gamma} \right) . \label{lst5} 
 \end{align}
Thus, let $\underline{u}:= \sum_{j=1}^{N} \alpha_{j} (\pi \delta_{a_{j},\lambda_{j}} )$, using \eqref{lst4},  we get
\begin{align} \label{eq:e4}
		\int_{\Omega} \left| u \right| ^{p-\varepsilon-1} u (\pi \delta_{a_{i},\lambda_{i}} ) = \int_{\Omega} \left( \underline{u} \right) ^{p-\varepsilon}  (\pi \delta_{a_{i},\lambda_{i}} )+(p-\varepsilon) \int_{\Omega} \left( \underline{u} \right) ^{p-\varepsilon-1}v (\pi \delta_{a_{i},\lambda_{i}} )+O\left( \left\| v \right\| ^{2} \right). 
	\end{align}
	The first integral in the right hand side of \eqref{eq:e4} can be written as (using  \eqref{lst5},  \eqref{lst3} and  \eqref{est0})
	\begin{align*}
	\int_{\Omega} \left( \underline{u} \right) ^{p-\varepsilon}  (\pi \delta_{a_{i},\lambda_{i}} ) & = \alpha_{i}^{p-\varepsilon}\int_{\Omega} (\pi \delta_{a_{i},\lambda_{i}} )^{p+1-\varepsilon}+\sum_{j \neq i}O\Big( \int_{\Omega}  \delta_{a_{i},\lambda_{i}} ^{p-\varepsilon} \delta_{a_{j},\lambda_{j}} + \int_{\Omega}  \delta_{a_{j},\lambda_{j}} ^{p-\varepsilon}  \delta_{a_{i},\lambda_{i}} \Big) \notag \\
	& = \alpha_{i}^{p-\varepsilon}\int_{\Omega}\delta_{a_{i},\lambda_{i}}^{p+1-\varepsilon} + O \Big( \int_{\Omega}\delta_{a_{i},\lambda_{i}}^{p-\varepsilon}\theta_{a_{i},\lambda_{i}}+\sum_{ j \neq i }\int_{\Omega} \delta_{a_{i},\lambda_{i}}^{p}\delta_{a_{j},\lambda_{j}}+\int_{\Omega} \delta_{a_{j},\lambda_{j}}^{p}\delta_{a_{i},\lambda_{i}} \Big) \notag \\
		& = \alpha_{i}^{p-\varepsilon} c_{0}^{-\varepsilon} \lambda_{i}^{-\varepsilon (n-2) /  2 } S_{n} +O \Big( \varepsilon + \frac{1}{\lambda_{i}^{n}}+T_{2}(\lambda_{i})+\sum_{i\neq j} \varepsilon_{ij} \Big) , 
	\end{align*}
where we have used $(64)$ of \cite{EM}, \eqref{eq:e3} and the last inequality in \eqref{dij}. Concerning the second integral in \eqref{eq:e4}, it is computed in Lemma \ref{vphi}.\\
	Combining the previous estimates, the proof follows.
\end{proof}

We will now present a balancing formula that involves the rate of concentration $\l_i$ and the self-interaction of the bubbles $\e_{ij}$. Specifically, we prove:
\begin{pro}\label{gradlam}
 Let $ n \geq 4$, $(\alpha,a,\lambda)\in \mathcal{O} (N,\mu)$, ${v}\in E_{a,\lambda }^{\perp }$ and $ u=\sum_{j=1}^{N}\alpha_{j}\pi \delta_{a_{j},\lambda_{j}}+v$. For $\varepsilon$ small and $i\in \left\{ 1, ... , N \right\}$, we have 
 \begin{align*}
 \Big \langle \nabla I_{V,\varepsilon }(u) ,\lambda_{i} &  \frac{\partial(\pi \delta_{a_{i} ,  \lambda _{i}})}{\partial \lambda_{i}}\Big \rangle  =  \sum_{j\neq i} \a_j\, \ov{c}_2 \l_i \frac{\partial \e_{ij}}{\partial \l_i} \Big(1- \a_i^{p-1-\e} \, \l_i^{ -\e (n-2) / {2}} - \a_j^{p-1-\e} \, \l_j^{ - \e (n-2) / {2}} \Big) \\
&  +  c_{0}^{-\varepsilon}\lambda_{i}^{-\varepsilon \frac{n-2}{2}}c_{2}\alpha_{i}^{p-\varepsilon}\varepsilon   
   - c(n) \alpha_{i}\frac{\ln^{\sigma_{n}}(\lambda_{i})}{\lambda_{i}^{2}} V(a_i) \left ( 2\alpha_{i}^{p-\varepsilon -1} c_{0}^{-\varepsilon }\lambda_{i}^{-\varepsilon \frac{n-2}{2}}-1 \right ) + O ( R_{\l_i} ), 
   \end{align*}
  where
\begin{align*}
 c(4) & :=  c_{0}^{2}\,mes(\mathbb{S}^{3}) \quad , \qquad c(n):=\frac{n-2}{2}c_{0}^{2}\int_{\mathbb{R}^{n}}\frac{|x|^{2}-1}{(1+|x|^{2})^{n-1}} dx >0\quad\mbox{if n}\geq 5, \\
  c_2 & := \Big(\frac{n-2}{2} \Big) ^{2} c_{0}^{\frac{2n}  { n-2 }} \int_ {\mathbb{R}^{n}} \frac{(\left|x\right|^{2}-1)}{(1+\left|x\right|^{2})^{n+1}} \ln(1+\left|x\right|^{2})dx  > 0 \quad ,  \quad \ov{c}_2   := c_0^{\frac{2n}{n-2}} \int_{\R^n } \frac{1}{ ( 1+ | x |^2 )^{(n+2)/2} }dx  , \\ 
  R_{\l_i} & :=   \varepsilon^{2}+\left\| v\right\|^{2}  + \frac{1}{\lambda_{i}^{n-2}}+\frac{1}{\lambda_{i}^{4}} + (\mbox{if}\; n=6)\frac{ \ln\lambda_{i}}{\lambda_{i}^{4}}   +  \sum_{k \neq j }\e_{kj}^{\frac{n}{n-2}} \ln (\e_{kj}^{-1}) \\
 & \qquad  \qquad \qquad +  \sum_{j \neq i } \Big( \frac{ 1} { (\l_j  \l_i ) ^{ (n-2) / 2 } } +\e_{ij} \Big( \Xi_{ij} +  \frac{ 1} { \l_j ^{3/2}} +  \frac{ 1} { \l_i ^{3/2}}  \Big) \Big)  
  \end{align*}
  and where $ \Xi_{ij} $ is defined in Lemma \ref{dj-thetai}.
\end{pro}
\begin{proof}
	We will follow the proof of Proposition $\ref{gradalpha}$ and we need to be more precise in some integrals. Since $v\in E^{\perp}_{a,\l}$, we have 
	$$\Big\langle u, \l_{i}\frac{\partial (\pi\d_{a_i, \l_i})}{\partial\l_{i}}\Big\rangle=\sum\alpha_{j} \Big \langle \pi\d_{a_j, \l_j},\l_{i}\frac{\partial (\pi\d_{a_i, \l_i})}{\partial\l_{i}} \Big\rangle .$$
	First,  for $j=i$, using Lemma 6.3 of \cite{2ABE}, we get
\begin{align}
\Big\langle \pi\d_{a_i, \l_i},\l_{i}\frac{\partial (\pi\d_{a_i, \l_i})}{\partial\l_{i}} \Big\rangle  & = \int_{\O}\d_{a_i, \l_i}^{\frac{n+2}{n-2}} \left(\l_{i}\frac{\partial\d_{a_i, \l_i}}{\partial\l_{i}}-\l_{i}\frac{\partial\th_{a_i, \l_i}}{\partial\l_{i}}\right) \nonumber \\
 & =  c(n) \frac{\ln^{\sigma_{n}}\l_{i}}{\l_{i}^{2}} V(a_i) + O\left(\frac{1}{\l_{i}^{n-2}}+\frac{1}{\l_{i}^{4}}+(\mbox{if } n = 6 )\frac{\ln\l_{i}}{\l_{i}^{4}}\right). \label{eqq:e2} 
\end{align}	
Second, for $j \neq i$, we have 
	\be \label{bn777}
\Big \langle \pi\d_{a_j, \l_j},\l_{i}\frac{\partial (\pi\d_{a_i, \l_i})}{\partial\l_{i}} \Big\rangle =  \int_{\O} \d_{a_j, \l_j}^{\frac{n+2}{n-2}} \Big( \l_{i}\frac{\partial  \d_{a_i, \l_i} }{\partial\l_{i}} - \l_i \frac{\partial  \th_{a_i, \l_i} }{\partial\l_{i}} \Big) . 
\ee 
Observe that, using Estimate F16 of \cite{Bahri-book} and the fact that $ d(a_k, \partial \O ) \geq c > 0 $ for $ k \in \{i,j\}$, we deduce that 
$$ \int_{\O} \d_{a_j, \l_j}^{\frac{n+2}{n-2}} \l_{i}\frac{\partial  \d_{a_i, \l_i} }{\partial\l_{i}} = \ov{c}_2 \l_i \frac{\partial \e_{ij} }{\partial \l_i} + O \Big( \e_{ij}^{\frac{n}{n-2}} \ln (\e_{ij}^{-1}) + \frac{ 1} { \l_j ^{(n+2) / 2 } \l_i ^{ (n-2) / 2 } } \Big) . $$
Furthermore, using Lemma \ref{dj-thetai}, equation \eqref{bn777} becomes
	\be \label{bn7}
\Big \langle \pi\d_{a_j, \l_j},\l_{i}\frac{\partial (\pi\d_{a_i, \l_i})}{\partial\l_{i}} \Big\rangle =  \ov{c}_2 \l_i \frac{\partial \e_{ij} }{\partial \l_i} + O \Big( \e_{ij}^{\frac{n}{n-2}} \ln (\e_{ij}^{-1}) + \frac{ 1} { (\l_j  \l_i ) ^{ (n-2) / 2 } } + \Xi_{ij}  \e_{ij} \Big) . 
\ee 
 Thus, we derive that
\begin{align}
\Big\langle u, \l_{i}\frac{\partial (\pi\d_{a_i, \l_i})}{\partial\l_{i}} \Big\rangle = &  \a_{i}c(n)\frac{\ln^{\sigma_{n}}\l_{i}}{\l_{i}^{2}} V(a_i) + 
\ov{c}_2 \sum_{ j \neq i } \alpha_j \l_i \frac{\partial \e_{ij} }{\partial \l_i}   + O\Big(\frac{1}{\l_{i}^{n-2}}+\frac{1}{\l_{i}^{4}} +  (\mbox{if } n = 6 )\frac{\ln\l_{i}}{\l_{i}^{4}}  \Big) \notag  \\
& +  \sum_{j \neq i} O  \Big( \e_{ij}^{\frac{n}{n-2}} \ln (\e_{ij}^{-1}) + \frac{ 1} { (\l_j  \l_i ) ^{ (n-2) / 2 } } +\e_{ij} \Xi_{ij}  \Big). \label{bn6}
\end{align}
Now, using  $\eqref{lst4}$, we get
	\begin{equation}\label{bn9}
	\int_{\O} |u|^{p-\e-1}u\l_{i}\frac{\partial (\pi\d_{a_i,\l_i})}{\partial\l_{i}}=\int_{\O}(\underline{u})^{p-\e}\l_{i}\frac{\partial (\pi\d_{a_i,\l_i})}{\partial\l_{i}}+(p-\e)\int_{\O}(\underline{u})^{p-\e-1}v\l_{i}\frac{\partial (\pi\d_{a_i,\l_i})}{\partial\l_{i}}+O(\|v\|^{2}).
	\end{equation}	
The last integral is computed in Lemma $\ref{vphi}$. For the first one, we need the following formula. 
Let $ 1 < \gamma \leq 3$,   $ b_j > 0 $ for  $ j \in \{ 1, \cdots, N\} $ and $ s $ be such that $ | s | \leq c b_i $ for some index $i$. Then it holds 
\be \label{lst6} \Big( \sum _{j=1}^N b_j \Big) ^{\gamma} s = \sum _{j=1}^N b_j ^{\gamma} s + \gamma b_i ^{\gamma -1 } \Big( \sum _{j \neq i} b_j \Big) s + \sum _{ k \neq j } O \Big( (b_k b_j) ^{ (\gamma + 1 )/2 } \Big)  . \ee
Thus, we get 
	\begin{align}\label{bn10}
\int_{\O} & (\underline{u})^{p-\e}\l_{i}\frac{\partial (\pi\d_{a_i,\l_i})}{\partial\l_{i}} =  \sum_{ 1 \leq j \leq N} \a_{j}^{p-\e}\int_{\O}(\pi\d_{a_j, \l_j})^{p-\e} \l_{i}\frac{\partial (\pi\d_{a_i,\l_i})}{\partial\l_{i}}  \notag  \\
 & + ( p-\e ) \sum_{  j \neq i}   \int_{\O} ( \a_i \pi\d_{a_i, \l_i})^{p-\e - 1} ( \a_j \pi\d_{a_j, \l_j})\l_{i}\frac{\partial (\pi\d_{a_i,\l_i})}{\partial\l_{i}}  +  \sum_{j\neq k}O \Big( \int_{\O} ( \d_{a_j,\l_j} \d_{a_k, \l_k} ) ^\frac{n}{n-2} \Big).
	\end{align}
 First, we remark that the remainder term is computed in \eqref{eq:f1}. 
 Second, the  first integral, with $ j = i$, is computed in \cite{2ABE} (see Equ. (45)) and we have 
	\begin{align} 
	\int_{\O}(\pi\d_{a_i, \l_i})^{p-\e}\l_{i}\frac{\partial(\pi\d_{a_i, \l_i})}{\partial\l_{i}} = & \,   c_{0}^{-\e}\l_{i}^{-\e\frac{n-2}{2}}\left(-c_{2}\e+2c(n)\frac{\ln^{\sigma_{n}}\l_{i}}{\l_{i}^{2}}\right)  \nonumber \\
 & + O\left(\e^{2}+\frac{1}{\l_{i}^{n-2}}+\frac{1}{\l_{i}^{4}}+(\mbox{if } n = 6 )\frac{\ln \l_{i}}{\l_{i}^{4}}\right). \label{bn11}
\end{align}
Third, we focus on estimating the first integral  of \eqref{bn10} with $ j \neq i$. Using  \eqref{lst3},  \eqref{est0} and \eqref{eq:8}, we deduce that 
\begin{align*}	
 \int_{\O} (\pi\d_{a_j, \l_j})^{p-\e} {\l_{i} } \frac{\partial(\pi\d_{a_i, \l_i})}{\partial \l _{i}}  & = \int_{\O}\Big( \d_{a_j, \l_j} ^ {p-\e} + O \Big(  \d_{a_j, \l_j}  ^{p-1} \theta_{a_j, \l_j} \Big) \Big) {\l_{i} } \frac{\partial \pi \d_{a_i, \l_i} }{\partial \l_{i}}   \\
& = c_0 ^{-\e} \l_j ^{- \e (n-2)/2} \int_\O \d_{a_j, \l_j}  ^p {\l_{i} } \frac{\partial \pi \d_{a_i, \l_i} }{\partial \l_{i}}  \\
 & + O \Big( \int_{ \O }  \theta_{a_j, \l_j} \d_{a_j, \l_j}  ^{ p -1 } \d_{a_i, \l_i}  + \e \int_{ \O} \d_{a_j, \l_j}  ^p  \ln(1+ \l_j ^2 | x-a_j|^2) \d_{a_i, \l_i}  \Big) .
\end{align*} 
Observe that, let $ B_j := B(a_j, d_0 ) $, we have 
 \begin{align} 
 & \int_{ \O } \d_{a_j, \l_j} ^{ p -1 } \theta_{a_j, \l_j}  \d_{a_i, \l_i}  \leq c  \int_{ B_j } R_1(x,a_j,\l_j) \d_{a_j, \l_j} ^{ p  }   \d_{a_i, \l_i} + \frac{c}{ \l_j ^{ (n+2)/2} } \int  \d_{a_i, \l_i} \notag \\
  & \leq c \frac{ \ln ^{\s_n} \l_j }{ \l_j ^2 }  \int_{ B_j } \d_{a_j, \l_j} ^p  \d_{a_i, \l_i} +  \int_{ B_j }  | x-a_j| ^2 | \ln | x-a_j ||^{\s_n} \d_{a_j, \l_j} ^p  \d_{a_i, \l_i} + \frac{c}{ \l_j ^{(n+2) / 2 }} \int  \d_{a_i, \l_i} . \label{wxc2} 
 \end{align} 
Using Holder's inequality and Lemma 2.2 of \cite{Dammak}, we deduce that
\begin{align*} \int_{ B_j } & | x-a_j| ^2  | \ln^{\s_n} | x-a_j || \d_{a_j, \l_j} ^p  \d_{a_i, \l_i}  \leq  \frac{c}{ \l_j ^{3/2} } \int_{ B_j } \frac{ | \ln | x-a_j ||^{\s_n}}{ | x-a_j|  } \d_{a_j, \l_j} ^\frac{n-1}{n-2}  \d_{a_i, \l_i}\\
 &  \leq  \frac{c}{ \l_j ^{3/2} } \Big( \int_{ B_j } \frac{ | \ln | x-a_j ||^{(n-1) \s_n}}{ | x-a_j| ^{n-1} } \Big) ^{1/(n-1)} \Big( \int_{ B_j }  \d_{a_j, \l_j} ^{\left(\frac{n-1}{n-2} \right)^2} \d_{a_i, \l_i}^\frac{n-1}{n-2} \Big) ^{(n-2) / (n-1)}  \leq \frac{c}{ \l_j ^{3/2} } \e_{ij} . 
 \end{align*}
Thus, \eqref{wxc2} becomes 
\be \label{wxc22}  \int_{ \O } \d_{a_j, \l_j} ^{ p -1 } \theta_{a_j, \l_j}  \d_{a_i, \l_i} \leq     \frac{ c }{ \l_j ^{3/2 }}  \e_{ij } + \frac{ c } {\l_i ^{( n-2) /2} \l_j ^{ (n + 2) / 2 } } . \ee
 Hence, we obtain 
 \begin{align*}	\int_{\O}  (\pi\d_{a_j, \l_j})^{p-\e} {\l_{i} } \frac{\partial(\pi\d_{a_i, \l_i})}{\partial \l _{i}}   =  & \,  c_0^{-\e} \l_j ^{ - \e \frac{n-2}{2} }  \langle \pi\d_{a_j, \l_j} ,   {\l_{i} } \frac{\partial(\pi\d_{a_i, \l_i})}{\partial \l _{i}}  \rangle  \\
  &  + O \Big( \frac{ 1 } {\l_i ^{ (n-2) / 2} \l_j ^{ (n + 2) / 2 } }  + \e_{ ij } \Big( \e   +  \frac{ 1 }{ \l_j ^{3/2} }  \Big)    \Big) . 
\end{align*} 
   In the same way, we deduce the estimate of the second integral of \eqref{bn10} and we get 
    \begin{align*}
     p	\int_{\O}   (\pi\d_{a_i, \l_i})^{p-\e-1}  {\l_{i} } \frac{\partial(\pi\d_{a_i, \l_i})}{\partial \l _{i}} \pi\d_{a_j,\l_j}  =  & \,  c_0^{-\e} \l_i ^{ - \e \frac{n-2}{2} }  \langle \pi\d_{a_j, \l_j} ,   {\l_{i} } \frac{\partial(\pi\d_{a_i, \l_i})}{\partial \l _{i}}  \rangle \\
  &   + O \Big( \frac{ 1 } {\l_j ^{ (n-2) / 2} \l_i ^{ (n + 2) / 2 } }    + \e_{ ij } \Big( \e   +  \frac{ 1 }{ \l_i ^{3/2} }  \Big)    \Big) . 
\end{align*}  	
Finally, combining $\eqref{bn6}$, $\eqref{bn9}$, $\eqref{bn10}$, $\eqref{bn7}$, $\eqref{bn11}$ and Lemma $\ref{vphi}$, the result follows.	
\end{proof}

Next, we present a balancing condition that involves the concentration points.
\begin{pro}\label{gradpoint}
 Let $ n \geq 4 $, $(\alpha,a,\lambda)\in \mathcal{O} (N,\mu)$, ${v}\in E_{a,\lambda }^{\perp }$ and $ u=\sum_{j=1}^{N}\alpha_{j} (\pi \delta_{a_{j},\lambda_{j}})+v$. For $\varepsilon$ small and $i\in \left\{ 1, ... , N \right\}$, we have 
 \begin{align*}
 \Big \langle \nabla I_{V,\varepsilon }(u),   \frac{1}{ \lambda_{i} } \frac{\partial(\pi \delta_{a_{i},\lambda _{i}})}{\partial a_{i}}\Big\rangle & =    c_2(n) \alpha_{i}\frac{\ln^{\sigma_{n}}(\lambda_{i})}{\lambda_{i}^{3}} \n V(a_i) \left ( 2\alpha_{i}^{p-\varepsilon -1} c_{0}^{-\varepsilon }\lambda_{i}^{-\varepsilon \frac{n-2}{2}}-1 \right ) \\
 & + \ov{c}_2 \sum_{ j \neq i } \a_j \frac{1}{ \l_{i} }  \frac{\partial \e_{ ij }}{\partial a_{i}} \Big( 1 - c_0^{-\e}  \l_i ^{ -\e \frac{n-2}{2} }  \a_i^{p-\e-1 } - c_0^{-\e}  \l_i ^{ -\e \frac{n-2}{2} } \a_j^{p-\e-1 }    \Big) + O ( R_{a_i}) ,
  \end{align*}
  with 
 \begin{align*}
  R_{a_i} :=  & \sum_{ j \neq i}  \l_j | a_i - a_j | \e_{ ij }^{\frac{n+1}{n-2}} + \left\| v\right\|^{2} + \frac{1}{\lambda_{i}^{4}} + (\mbox{if}\; n=5)\frac{ \ln\lambda_{i}}{\lambda_{i}^{4}}  \\
   &  +\sum  \frac{1}{\lambda_{k}^{n-1}}  + \sum_{k \neq j } \e_{kj}^{n/(n-2} \ln \e_{kj}^{-1}  + \sum  \e_{ij}  \Big( \e + \frac{1}{\l_i} + \frac{1}{\lambda_{j}^{3/2}} \Big) + \e\frac{\ln^{\sigma_{n}} \l_{i}}{\l_{i}^{2}} 
  \end{align*}
where  
$$ 
   c_2(4) := \frac{1}{2} c_{0}^{2}\, meas(\mathbb{S}^{3}) \qquad \mbox{ and }   \qquad c_2(n):=\frac{n-2}{n}c_{0}^{2}\int_{\mathbb{R}^{n}}\frac{|x|^{2}}{(1+|x|^{2})^{n-1}} \quad\mbox{if n}\geq 5.
$$
\end{pro}

\begin{proof}
We will follow the proof of Proposition $\ref{gradlam}$. Since $v\in E^{\perp}_{a,\l}$, we have 
	$$ \Big\langle u, \frac{1}{ \l_{i} } \frac{\partial (\pi\d_{a_i, \l_i})}{\partial a_{i}} \Big\rangle = \sum \alpha_{j} \Big\langle \pi\d_{a_j, \l_j}, \frac{1}{ \l_{i} } \frac{\partial (\pi\d_{a_i, \l_i})}{\partial a_{i}} \Big\rangle .$$
First, the scalar product, with $ j = i $, is computed in \cite{2ABE}  and we have
\be \label{eqq:e3} 
 \Big\langle \pi\d_{a_i, \l_i}, \frac{1}{ \l_{i} } \frac{\partial (\pi\d_{a_i, \l_i})}{\partial a_{i}} \Big\rangle = - c_2(n)\frac{\ln^{\sigma_{n}}\l_{i}}{\l_{i}^{3}} \n V(a_i) + O \left( \frac{1}{\l_{i}^{n-1}}+\frac{1}{\l_{i}^{4}}+(\mbox{if }n =5)\frac{\ln\l_{i}}{\l_{i}^{4}}\right). 
 \ee
Second, for $j \neq i$,  using  \eqref{est2} and \eqref{est3},  the last inequality in \eqref{dij} and  (F11) of \cite{Bahri-book}, we get 
	\begin{align}
 \Big \langle \pi\d_{a_j, \l_j}, \frac{1}{ \l_{i} } \frac{\partial (\pi\d_{a_i, \l_i})}{\partial a_{i}} \Big \rangle & = \int_{\O} \d_{a_j, \l_j}^{\frac{n+2}{n-2}}  \Big( \frac{1}{ \l_{i} } \frac{\partial \d_{a_i, \l_i}}{\partial a_{i}}  -  \frac{1}{ \l_{i} } \frac{\partial \theta_{a_i, \l_i}}{\partial a_{i}} \Big) \nonumber \\
 & = \int_{\R^n} \d_{a_j, \l_j}^{\frac{n+2}{n-2}}  \frac{1}{ \l_{i} } \frac{\partial \d_{a_i, \l_i}}{\partial a_{i}} + O \Big( \frac{1}{\l_j ^{(n+2) / 2 } \l_i ^{ n /2}} +  \frac{1}{\l_i } \int_{\O} \d_{a_j, \l_j}^{\frac{n+2}{n-2}}  \delta_{a_i, \l_i}  \Big) \notag \\
  & = \,  \ov{c}_2 \frac{1}{ \l_{i} }  \frac{\partial \e_{ ij }}{\partial a_{i}}  + O \Big( \l_j | a_i - a_j | \e_{ ij }^{\frac{n+1}{n-2}} + \frac{ 1 }{ \l_{i}}   \e_{ij} + \frac{ 1 } {\l_i ^{n/2} \l_j ^{ (n + 2) / 2 } } \Big) . \label{wxc1} 
 \end{align}
Thus, we derive that
\begin{align}\label{bn6p}
 \Big\langle & u, \frac{1}{\l_{i} }\frac{\partial (\pi\d_{a_i, \l_i})}{\partial a_{i}} \Big\rangle = 
 - \a_{i} c_2(n) \frac{\ln^{\sigma_{n}}\l_{i}}{\l_{i}^{3}} \n V(a_i) +  \ov{c}_2 \a_j \frac{1}{ \l_{i} }  \frac{\partial \e_{ ij }}{\partial a_{i}}  \\ 
  &  +   O\Big( \sum_{j\neq i} \Big(  \l_j | a_i - a_j | \e_{ ij }^{\frac{n+1}{n-2}} + \frac{ 1}{ \l_{i} } \e_{ ij } + \frac{ 1 } {\l_i ^{n/2} \l_j ^{ (n + 2) / 2 } } \Big)  + \frac{1}{\l_{ i }^{n - 1}}  +  \frac{1}{\l_{i}^{4}} + (\mbox{if } n =5)\frac{\ln\l_{i}}{\l_{i}^{4}}   \Big). \nonumber 
\end{align}
Now, notice that \eqref{bn9} and \eqref{bn10} hold true with $ \l_i ^{-1} \partial ( \pi \d_{a_i, \l_i}) / \partial a_i $ instead of $ \l_i  \partial ( \pi \d_{a_i, \l_i}) / \partial \l_i $. Furthermore, Equ. (53) of \cite{2ABE} implies that 
  	\begin{align}\label{bn11p}
	\int_{\O}(\pi\d_{a_i, \l_i})^{p-\e} \frac{1}{\l_{i} } \frac{\partial(\pi\d_{a_i, \l_i})}{\partial a_{i}} =  & - 2 c_2(n)  c_{0}^{-\e}\l_{i}^{-\e\frac{n-2}{2}} \frac{\ln^{\sigma_{n}} \l_{i}}{\l_{i}^{3}} \n V(a_i) \nonumber	\\	
	& + O\left( \frac{1}{\l_{i}^{n-1}}+\frac{1}{\l_{i}^{4}}+(\mbox{if } n=5)\frac{\ln \l_{i}}{\l_{i}^{4}} + \e\frac{\ln^{\sigma_{n}} \l_{i}}{\l_{i}^{2}}\right).
	\end{align}
In addition, for $ j \neq i $, using  \eqref{lst3},  \eqref{est0}, \eqref{wxc22} and \eqref{eq:8}, we deduce that 
\begin{align*}	
  & \int_{\O} (\pi\d_{a_j, \l_j})^{p-\e} \frac{1}{\l_{i} } \frac{\partial(\pi\d_{a_i, \l_i})}{\partial a _{i}}   = \int_{\O}\Big( \d_{a_j, \l_j} ^ {p-\e} + O \Big(  \d_{a_j, \l_j}  ^{p-1} \theta_{a_j, \l_j} \Big) \Big) \frac{1}{\l_{i} } \frac{\partial \pi \d_{a_i, \l_i} }{\partial a_{i}}   \\
& = c_0 ^{-\e} \l_j ^{- \e \frac{n-2}{2}} \int_\O \d_{a_j, \l_j}  ^p \frac{1}{\l_{i} } \frac{\partial \pi \d_{a_i, \l_i} }{\partial a_{i}} + O \Big( \int_{ \O }  \theta_{a_j, \l_j} \d_{a_j, \l_j}  ^{ p -1 } \d_{a_i, \l_i}  + \e \int_{ \O} \d_{a_j, \l_j}  ^p  \ln(1+ \l_j ^2 | x-a_j|^2) \d_{a_i, \l_i}  \Big) \\
 & = c_0^{-\e} \l_j ^{ - \e \frac{n-2}{2} }  \langle \pi\d_{a_j, \l_j} ,   \frac{1}{\l_{i} } \frac{\partial(\pi\d_{a_i, \l_i})}{\partial a _{i}}  \rangle  + O \Big( \frac{ 1 } {\l_i ^{ (n-2) / 2} \l_j ^{ (n + 2) / 2 } } + \e_{ij} \Big( \e + \frac{1}{ \l_j ^{3/2}} \Big) \Big)
\end{align*} 

   In the same way, we obtain 
    \begin{align*}
     p	\int_{\O}   (\pi\d_{a_i, \l_i})^{p-\e-1}  \frac{1}{\l_{i} } \frac{\partial(\pi\d_{a_i, \l_i})}{\partial a _{i}} \pi\d_{a_j,\l_j} & =   c_0^{-\e} \l_i ^{ - \e \frac{n-2}{2} }  \langle \pi\d_{a_j, \l_j} ,   \frac{1}{\l_{i} } \frac{\partial(\pi\d_{a_i, \l_i})}{\partial a _{i}}  \rangle \\
  &   + O \Big( \frac{ 1 } {\l_j ^{ (n-2) / 2} \l_i ^{ (n + 2) / 2 } }    + \e_{ ij } \Big( \e   +  \frac{ 1 }{ \l_i ^{3/}2 }  \Big) \Big)   . 
\end{align*}  
Combining the previous estimates, the result follows.	
\end{proof}

\section{Interior blowing up solutions with clustered bubbles}

This section is devoted to the proof of Theorems \ref{th:t1} and \ref{th:t2}. We start by proving Theorem \ref{th:t1}. The proof structure adopts a method comparable to the one presented in \cite{BLR, EM, BEM}. Let $n\geq 4$,  $b\in\O$ be a non-degenerate critical point of $ V $ and  $(\overline{z}_1,\cdots , \overline{z}_N)$ be  a non-degenerate critical point of $F_{b,N}$, where $F_{b, N}$ is defined by \eqref{FYN}. We start by introducing a parameterization of the
neighborhood of the desired constructed solutions. Let
\begin{align}\label{e410}
\mathcal{O}_1    (N,b,\e) & =  \lbrace  (\alpha, \lambda, a, v) \in (\mathbb{R}_+)^N \times (\mathbb{R}_+)^N \times \Omega^N \times H^1(\Omega): v \in E_{a,\lambda} ^\perp,  \quad\| v \| < \sqrt{\varepsilon} ,  \nonumber \\
& \vert\alpha_i -1\vert <  \varepsilon \ln^2 \varepsilon,\, \frac{1}{c} < \frac{\ln ^{\s_n} (\l_i ) }{\lambda^2_i \varepsilon } < c,  \vert a_i - b - \eta ( \e ) \sigma\overline{z}_i \vert \leq  \,\eta ( \e ) \eta_0 \quad\forall 1\leq i \leq N \rbrace, 
\end{align}
where $\eta_0 $ is a small constant,  $ \s_4 = 1$, $\s_n = 0 $ for $ n \geq 5$,  $E_{a,\lambda} $ is defined by \eqref{Eal}, 
\be\label{sigma}
\eta ( \e ) := \begin{cases} 
\e^{\gamma} & \mbox{ if } n \geq 5 \\
| \ln \e | ^{-1/4} & \mbox{ if } n = 4 \end{cases} , \quad 
\gamma=\frac{n-4}{2n}\quad\mbox{and}\quad \sigma = \left(\frac{\ov{c}_2}{c_2(n)}\right)^{1/n} \left(\frac{c_2}{c(n) V ( b )}\right)^\gamma
\ee
with $c_2$, $\ov{c}_2$, $c_2(n)$ and $c(n)$ are the constants defined in Propositions \ref{gradlam} and  \ref{gradpoint}. \\
 In addition, we consider the following function 
$$ \psi_{\varepsilon}: \mathcal{O}_1(N,b,\e) \rightarrow \mathbb{R}, \quad(\alpha, \lambda, a, v) \mapsto \psi_\varepsilon (\alpha, \lambda, a, v) = I_{V,\varepsilon} \left( \sum_{i=1}^{N}\alpha_i \pi\delta_{a_i,\lambda_i} +v\right).
  $$
We notice that $(\alpha, \lambda, a, v)$ is a critical point of $\psi_\varepsilon$ if and only if
 $u=\sum_{i=1}^{N}\alpha_i \pi \d_{a_i, \l_i}  +v$ is a critical point of $I_{V,\varepsilon}$. Thus, we need to look for critical points of  $\psi_\varepsilon$. Since the variable $v$ belongs to $E_{a,\lambda}^{\perp}$, the Lagrange multiplier theorem allows us to get the following proposition.

\begin{pro}\label{p:51} $(\alpha, \lambda, a, v)\in \mathcal{O}_1(N,b,\e)$ 
is a critical point of $\psi_\varepsilon$ if and only if there exists $(A,B,C) \in \mathbb{R}^N \times \mathbb{R}^N \times (\mathbb{R}^n)^N$ such that the following holds: 
\begin{align}
& \frac{\partial \psi_{\varepsilon}}{\partial \alpha_i}(\alpha, \lambda, a, v) = 0  \, \, \forall \, \,  i\in \left\lbrace 1,...,N\right\rbrace ,  \label{Ealphai}\\
& \frac{\partial \psi_{\varepsilon}}{\partial \lambda_i}(\alpha, \lambda, a, v) = B_i \langle v ,  \l_i \frac{\partial^2 (\pi \delta_{a_i,\l_i}) }{\partial \lambda_i^2} \rangle  + \sum_{j=1}^n C_{i j}  \langle v ,  \frac{1}{\l_i} \frac{\partial^2 (\pi \delta_{a_i,\l_i})}{\partial \lambda_{i} \partial a_{i}^{j}} \rangle 
 \, \, \forall \, \,  i \in \left\lbrace 1,...,N\right\rbrace ,  \label{Elambdai} \\
& \frac{\partial \psi_{\varepsilon}}{\partial a_i}(\alpha, \lambda, a, v) = B_i \langle v , \l_i \frac{\partial^2 (\pi \delta_{a_i,\l_i})}{\partial \lambda_i\partial a_i } \rangle + \sum_{j=1}^n C_{i j}  \langle v , \frac{1}{\l_i} \frac{\partial^2 (\pi \delta_{a_i,\l_i}) }{ \partial a_{i}^{j}\partial a_{i}} \rangle 
 \, \, \forall \, \,  i\in \left\lbrace 1,...,N\right\rbrace ,  \label{Eai}\\
&\frac{\partial \psi_{\varepsilon}}{\partial v}(\alpha, \lambda, a, v) = \sum_{k=1}^N \Big( A_{k} \pi \d_{a_k, \l_k} 
  +B_{k} \l_k \frac{\partial (\pi \d_{a_k,\l_k} ) }{\partial \l_k } + \sum_{j=1}^n C_{kj} \frac{1}{\l_k} \frac{\partial (\pi \d_{a_k,\l_k} )}{\partial a_k^j} \Big) . \label{Ev}
\end{align}
\end{pro}

To prove Theorem \ref{th:t1}, we will make a careful study of the previous equations. Notice that
\begin{align}
&\frac{\partial \psi_{\varepsilon}}{\partial \alpha_i}(\alpha, \lambda, a, v) = \Big\langle \n  I_{V, \varepsilon} \Big( \sum_{i=1}^{N}\alpha_{i}\pi \delta_{a_i,\l_i} + v \Big), \pi \delta_{a_i,\l_i}   \Big\rangle ,  \label{15i} \\
&\frac{\partial \psi_{\varepsilon}}{\partial \lambda_i}(\alpha, \lambda, a, v) = \Big\langle \n  I_{V, \varepsilon} \Big( \sum_{i=1}^{N}\alpha_{i}\pi \delta_{a_i,\l_i} + v \Big),\alpha_{i}\frac{\partial ( \pi \delta_{a_i,\l_i} ) }{\partial\lambda_{i}}   \Big\rangle , \label{16i} \\
&\frac{\partial \psi_{\varepsilon}}{\partial a_i}(\alpha, \lambda, a, v) = \Big\langle  \n I_{V, \varepsilon} \Big( \sum_{i=1}^{N}\alpha_{i}\pi \delta_{a_i,\l_i} + v \Big) , \alpha_{i}\frac{\partial (\pi \delta_{a_i,\l_i}) }{\partial a_{i}}   \Big\rangle ,  \label{17i} \\
&\frac{\partial \psi_{\varepsilon}}{\partial v}(\alpha, \lambda, a, v)= \n I_{V, \varepsilon} \Big( \sum_{i=1}^{N}\alpha_{i} \pi \delta_ {a_i,\l_i}+ v \Big) .\label{18i}
\end{align}
Let $(\a, \l, a, 0) \in \mathcal{O}_1(N,b,\e)$, which is defined by \eqref{e410}. We are going to solve the system defined by equations \eqref{Ealphai}, \eqref{Elambdai}, \eqref{Eai} and \eqref{Ev}. The strategy of the proof of Theorem \ref{th:t1} is to reduce the problem to a finite dimensional system. Proposition \ref{v bar} allows us to obtain such a reduction by finding $\overline{v}$ verifying the equation \eqref{Ev}.
 Combining equations \eqref{Ealphai}, ..., \eqref{18i}, we see that $u=\sum_{i=1}^N\a_i \pi \d_{a_i , \l_i } + \ov{ v } $ is a critical point of $I_{V, \e} $ if and only if $(\a,\l, a)$ solves the following system: for each $1\leq i\leq N$
\begin{align}
&\left(E_{\alpha_i}\right) \quad \Big \langle \n  I_{V , \varepsilon} (u), \pi \delta_{a_i, \l_i}  \Big\rangle ) = 0 ,   \label{97i}\\
&\left(E_{\lambda_i}\right)\quad \Big\langle \n  I_{V , \varepsilon} ( u),\alpha_{i}\frac{\partial (\pi \delta_{a_i, \l_i}) }{\partial \l_{i}}  
 \Big\rangle = B_i \Big \langle  \ov{ v } , \l_i \frac{\partial^2 ( \pi \delta_ {a_i, \l_i} ) }{\partial \lambda_i^2} \Big \rangle + \sum_{j=1}^n C_{i j} 
\Big\langle \ov{ v } , \frac{1}{\l_i} \frac{\partial^2 ( \pi \delta_{a_i, \l_i} ) }{\partial \lambda_{i} \partial a_{i}^{j}} \Big\rangle  , 
\label{98i} \\
&\left(E_{a_i}\right) \quad \Big \langle  \n I_{ V , \varepsilon} ( u),\alpha_{i} \frac{\partial ( \pi \delta_{a_i, \l_i}) }{\partial a_{i}}  
 \Big\rangle = B_i \Big\langle \ov{ v } , \l_i \frac{\partial^2  ( \pi \delta_{a_i, \l_i}) }{\partial \lambda_i\partial a_i } \Big\rangle + \sum_{j=1}^n C_{i j} 
\Big\langle  \ov{ v } , \frac{1}{\l_i} \frac{\partial^2 ( \pi \delta_{a_i, \l_i}) }{ \partial a_{i}^{j}\partial a_{i}} \Big\rangle .
 \label{99i}
\end{align}
We are looking for $(\a, \l, a) \in \mathcal{O}_2 (N,b,\e)$ solution of the system defined by equations \eqref{97i}, \eqref{98i} and \eqref{99i}, where $ \mathcal{O}_2 (N,b,\e)$ is defined by
\be\label{e411}
\mathcal{O}_2  (N,b,\e ) =\lbrace (\alpha, \lambda, a) \in (\mathbb{R}_+)^N \times (\mathbb{R}_+)^N \times \Omega^N : (\alpha, \lambda, a, 0) \in \mathcal{O}_1(N,b,\e)  \rbrace. 
\ee
In order to work with a simpler system, we make the following change of variables, for $ 1\leq i \leq N $,
\be\label{p2-1bis}
\beta_i=1-\a_i^{p-1}, \quad ( a_i-b)= \eta ( \e ) \sigma (\z_i +\ov{z}_i) , \quad \frac{ 1 }{\l_i ^2 }= \Big( \frac{c_2}{c(n)}\frac{1}{V(b)}\,\e (  | \ln \e | / 2 ) ^{-\s_n} \Big)   (1+\land_i) . 
\ee
Using these changes of variables, it is easy to see that
\begin{align}
& c \leq \l_i | a_i - a_j | \e ^{ 2/n } | \ln \e |^{ - \s_n / 4 }  \leq c' , \\
&\e_{ij} = \frac{1}{(\l_i\l_j|a_i-a_j|^2)^{\frac{n-2}{2}}} \Big( 1 + O \Big( \frac{1}{\l_i ^2 | a_i - a_j | ^2 } +  \frac{1}{\l_j ^2 | a_i - a_j | ^2 } \Big) \Big) \notag \\
 & \quad =   \frac{1}{(\l_i\l_j|a_i-a_j|^2)^{\frac{n-2}{2}}}(1+O( \e ^{ 4 / n } | \ln \e |^{ -\s_n / 2 }  )) = O\left(\e^{2\g +1} | \ln \e |^{ - \s_n / 2 }   \right) , \label{S1}
 \end{align} 
 \begin{align}
\frac{\partial \e_{ij}}{\partial a_i} & = ( n-2 ) \l_i \l_j (a_j - a_i ) \e_{ij}^{ n/(n-2) }  =  \frac{ (n-2)( a_j - a_i ) } { ( \l_i \l_j) ^{(n-2)/2 } | a_i -a_j | ^{n} } \Big( 1+  O( \e ^{ 4 / n } | \ln \e |^{ -\s_n / 2 }  ) \Big) \notag \\
 & \quad = (n-2) \frac{ ( \e | \ln \e | ^{-\s_n} ) ^{(n-2)/2} } { \sigma^{n-1} \eta(\e)^{n-1} }\Big( \frac{ c_2} { c (n) V (b) } \Big)^{\frac{n-2}{2}} \frac{(\z_j +\bar{z}_j -\z_i -\bar{z}_i)}{\left|\z_i +\bar{z}_i -\z_j -\bar{z}_j\right|^n} \times  \label{S22} \\
  & \qquad \qquad \qquad \qquad \qquad \qquad  \qquad \times \Big(1+R(\land_i, \land_j ) \Big)  \Big( 1+  O( \e ^{ 4 / n } | \ln \e |^{ -\s_n / 2 }  ) \Big)   \notag \\
  &  = O \Big( \e^{\g +1} | \ln \e | ^{ - \s_n / 4 } \Big) , \label{S2}
\end{align}
where  
\be\label{S3}
R( \land_i, \land_j)= \frac{n-2}{4}\land_i + \frac{n-2}{4}\land_j + O\left(\land_i^2\right) +O\left(\land_j^2\right).
\ee 

Next, using Propositions  \ref{lfv}, \ref{gradalpha}, \ref{gradlam} and \ref{gradpoint},   we deduce  that the following estimates hold:
\begin{lem}\label{lS4}
For $\e$ small, the following statements hold:
\begin{align*}
&\| \ov{v} \| \leq \, c\,  \e  \, \, ,  \qquad R_{\alpha_i} \leq \,  c \,  \e  ,\\
 & R_{\l_i} \,\leq \, c \, \e^{2\g +1} | \ln \e | ^{ -\s_n /2} , 
\quad 
 R_{a_i} \leq c \begin{cases} \e^2 | \ln \e | + \e^{ (5/2) - ( 4/n) } & \mbox{ if } n \geq 5, \\
                                               \e ^{3/2} | \ln \e |^{-1} & \mbox{ if } n = 4 . 
    \end{cases}
\end{align*}
where $R_{\a_i}$, $R_{\l_i}$ and $R_{a_i}$ are defined in Propositions \ref{gradalpha}, \ref{gradlam} and \ref{gradpoint} respectively.
\end{lem}
 Now, arguing as in the proof of Lemma 4.2 of \cite {2ABE} and using Lemma \ref{ajout1}, we derive that  the constants $A_i$'s, $B_i$'s and $C_{ij}$'s which appear in equations $(E_{v})$, $(E_{\l_i})$ and $(E_{a_i})$ satisfy the following estimates:
\begin{lem}\label{lS5}
Let  $(\a, \l, a )\in \mathcal{O}_2(N, b ,\e)$. Then, for $\e$ small, the following statements hold:
$$
A_i= O(\e \ln^2 \e ),\quad B_i=O(\e)\quad\mbox{and}\quad  | C_{ij} | \leq c \begin{cases} \e^{\g + 3/2} & \mbox{ if } n \geq 5, \\
    \e ^{3/2} | \ln \e | ^{-3/4} & \mbox{ if } n = 4, \end{cases}
    \quad \forall\,  i\leq N, \, \forall\, j \leq n.
$$
\end{lem} 
 Next, our aim is to rewrite equations $(E_{\a_i})$, $(E_{\l_i})$ and  $(E_{a_i})$ in a simple form.
\begin{lem}\label{lS6}
For $\e$ small,  equations $(E_{\a_i})$, $(E_{\l_i})$ and  $(E_{a_i})$ are equivalent to the following system
$$
(S')\quad
\begin{cases}
\beta_i=O(\e|\ln \e|)\quad\forall\, 1\leq i\leq N , & (F' _{\a_i}) \\
\land_i=O\left(|\beta_i|+\e^{2\g} | \ln \e |^{ - \s_n / 2 } \right) \quad\forall\, 1\leq i\leq N , &  (F' _{\l_i}) \\
D^2 V(b)(\z_i,.)-(n-2)\sum_{j\neq i}\frac{\z_j-\z_i}{|\ov{z}_j-\ov{z}_i|^n} +n(n-2)\sum_{j\neq i} \Big\langle \frac{\ov{z}_j-\ov{z}_i}{|\ov{z}_j-\ov{z}_i|^2}, \z_j-\z_i \Big\rangle \frac{\ov{z}_j-\ov{z}_i}{|\ov{z}_j-\ov{z}_i|^n}   \\
\qquad \qquad \qquad -(n-2)\sum_{j\neq i}\frac{\ov{z}_j-\ov{z}_i}{|\ov{z}_j-\ov{z}_i|^n}\left(\frac{n-6}{4}\land_i + \frac{n-2}{4}\land_j\right) =O(R'_S)
 \quad\forall \,1\leq i\leq N, &  (F' _{a_i})
\end{cases}
$$
where 
$$
R'_S = \sum_{j=1}^N \land_j^2 + \sum_{j=1}^N |\z_j|^2 + \eta (\e) + \e^{\frac{1}{2}-\g}|\ln \e | .
$$
\end{lem}
\begin{pf}
First, using Proposition \ref{gradalpha}, Lemma \ref{lS4} and the fact that $(\a,\l,a)\in \mathcal{O}_2 (N,b,\e) $, we see that  $(E_{\a_i})$ is equivalent to the equation  $ (F'_{\a_i})$. 
Second, using Lemmas \ref{lS5} and \ref{ajout1}, we write
$$
 \Big \langle \n I_\e (u), \l_i \frac{\partial \pi \d_{a_i, \l_i}}{\partial\l_i} \Big\rangle = O\left(|B_i| \| v\| \left\| \l_i ^ 2  \frac{\partial^2 \pi \d_{a_i, \l_i}}{\partial\l_i^2} \right\| + \sum_{j=1}^N |C_{ij}|  \| v\|  \left \|  \frac{\partial^2 \pi \d_{a_i, \l_i}}{\partial\l_i \partial a_i^j} \right\|  \right) = O\left(\e \| v\|\right).
$$
Using Proposition \ref{gradlam},  Lemma \ref{lS4} and \eqref{S1},  we obtain
$$
c_2\,\e - c (n) \frac{ \ln ^{\s_n} \l_i }{\l_i^2} V(a_i) = O\left(\e |\beta_i| + \e ^{2\g +1} | \ln \e | ^{ - \s_n /2} \right).
$$
But we have
$$
V(a_i) = V (b) + O\left(|a_i-b|^2\right) = V (b)+ O\left(\eta (\e) ^{2}\right).
$$
This implies that $(E_{\l_i})$ is equivalent to the equation  $(F'_{\l_i})$. \\
To deal with the third equation $(E_{a_i})$, using Lemma \ref{ajout1}, we write 
\begin{align}\label{S7}
\Big\langle \n I_\e (u), \frac{1}{\l_i} \frac{\partial \pi \d_{a_i, \l_i}}{\partial a_i}  \Big\rangle & = O\left(|B_i| \| v\| \left\|  \frac{\partial^2 \pi \d_{a_i, \l_i}}{\partial\l_i \partial a_i} \right\| + \sum_{j=1}^N |C_{ij}|  \| v\|  \left\| \frac{1}{\l_i^2} \frac{\partial^2 \pi \d_{a_i, \l_i}}{\partial a_i \partial a_i^j}\right\| \right) = O\big(\e\| v\| \big).
\end{align}
 But, combining \eqref{S22}, Proposition \ref{gradpoint} and  Lemma \ref{lS4}, \eqref{S7} becomes
 \begin{align}\label{S8}
 c_2(n) \frac{ \ln ^{\s_n} \l_i }{\l_i^3}\n V(a_i)& - (n - 2) \ov{c}_{10} \frac{\e \eta (\e) }{\l_i} \sum_{j\neq i} \frac{(\z_j+\bar{z}_j - \z_i -\bar{z}_i)}{|\z_j+\bar{z}_j - \z_i -\bar{z}_i|^n}\left(1+\frac{n-2}{4}\land_i + \frac{n-2}{4}\land_j \right)\nonumber\\
 &  = O \Big(\e^2|\ln \e |^{\s_n /2 } + R_{a_i} + \eta (\e)  \e^{3/2} | \ln \e | ^{ -\s_n / 2 } \sum \Big( | \beta_j | + \land_j^2 \Big) \Big),
\end{align}
where
$$
\ov{c}_{10}=\frac{\ov{c}_2}{\sigma^{n-1}}\left(\frac{c_2}{c (n)  V(b)}\right)^{ ( n-2 ) / 2 }.
$$
Observe that
\be \label{S9}
\n V (a_i)  = D^2 V ( b )( a_i - b , . ) + O\left( | a_i - b |^2 \right) = \eta(\e)  \sigma D^2 V ( b ) (\z_i +\bar{z}_i , . ) + O\left( \eta(\e)^{2 }\right).
\ee
Combining \eqref{S8}, \eqref{S9} and \eqref{sigma}, we obtain
\begin{align}\label{S10}
D^2 V ( b ) (\z_i+\bar{z}_i,.) & - (n-2)\sum_{j\neq i} \frac{(\z_j+\bar{z}_j - \z_i -\bar{z}_i)}{|\z_j+\bar{z}_j - \z_i -\bar{z}_i|^n}\left(1+\frac{n-6}{4}\land_i + \frac{n-2}{4}\land_j \right)\nonumber\\
&= O\left( \eta (\e) + \e^{\frac{1}{2}-\g} | \ln \e| + \sum\Big( | \beta_j | +  \land_j^2 \Big) \right)  .
\end{align}
But we have
\be\label{S11}
D^2 V ( b ) (\z_i+\bar{z}_i,.) = D^2 V ( b ) (\z_i,.)+ D^2 V ( b ) (\bar{z}_i,.)
\ee
and
\begin{align}\label{S12}
&\frac{(\z_j+\bar{z}_j - \z_i -\bar{z}_i)}{|\z_j+\bar{z}_j - \z_i -\bar{z}_i|^n}=\left(\frac{(\bar{z}_j -\bar{z}_i)}{|\bar{z}_j -\bar{z}_i|^n} + \frac{(\z_j - \z_i)}{|\bar{z}_j -\bar{z}_i|^n}\right)\left[1-n \left\langle \frac{(\bar{z}_j -\bar{z}_i)}{|\bar{z}_j -\bar{z}_i|^2}, \z_j-\z_i\right\rangle +O\left( | \z_i |^2 + | \z_j |^2\right)\right]\nonumber\\
&\qquad\qquad=\frac{(\bar{z}_j -\bar{z}_i)}{|\bar{z}_j -\bar{z}_i|^n}-n \left\langle \frac{(\bar{z}_j -\bar{z}_i)}{|\bar{z}_j -\bar{z}_i|^2}, \z_j-\z_i\right\rangle \frac{(\bar{z}_j -\bar{z}_i)}{|\bar{z}_j -\bar{z}_i|^n}+\frac{(\z_j - \z_i)}{|\bar{z}_j -\bar{z}_i|^n} + O\left( | \z_i |^2 + | \z_j |^2\right).
\end{align}
Combining \eqref{S10}, \eqref{S11}, \eqref{S12} and the fact that $(\bar{z}_1,..., \bar{z}_N) $ is a critical point of $F_{N,b}$, we  see that equation  $(E_{a_i})$ is equivalent to the equation 
$ (F'_{a_i}) $  which completes the proof of Lemma \ref{lS6}.
\end{pf} 

Now, we are ready to prove our results related to the construction of clustered bubbling  solutions.

\begin{pfn} {\bf of Theorem \ref{th:t1}}
Note that the system $(F'_{a_1})$, ..., $(F'_{a_N})$ is equivalent to 
$$ 
\frac{1}{2}D^2 F_{N,b}(\ov{z}_1,..., \ov{z}_N)(\z_1,..., \z_N) - (n-2) ( \G_1, \cdots , \G_N) = O\left(\sum_{j=1}^N ( \land_j^2 +  |\z_j|^2) + \eta (\e) + \e^{\frac{1}{2}-\g}|\ln \e | \right) , $$
where 
$$ \G_i :=  \sum_{j\neq i}\frac{\ov{z}_j-\ov{z}_i}{|\ov{z}_j-\ov{z}_i|^n}\left(\frac{n-6}{4}\land_i + \frac{n-2}{4}\land_j\right) \quad \forall \, \, 1 \leq i \leq N . 
$$
As in the proof of Theorem $1$ of \cite{EM}, we define a linear map by taking the left hand side of the system defined by $(F'_{\a_i})$, $(F'_{\l_i})$ and $(F'_{a_i})$. Since $(\bar{z}_1,..., \bar{z}_N)$ is a non-degenerate critical point of $F_{N,b}$, we deduce that   such a linear map is invertible and arguing as in the proof of Theorem $1$ of \cite{EM}, we derive that the system $(S')$ has a solution $(\beta_\e, \land_\e, \z_\e)$ for $\e$ small. This implies that  $(\mathcal{P}_{V,\e})$ admits a solution $ u_{\e,b} =\sum_{i=1}^N \a_{i,\e} \pi \d_{{a_{i,\e}}, \l_{i,\e}} +v_\e $
 and, by construction, properties \eqref{t30} - \eqref{t33} are satisfied. The proof of Theorem \ref{th:t1} is thereby completed.
\end{pfn}

\begin{pfn} {\bf of Theorem \ref{th:t2}}
Let $b_1$, $b_2$ be two non-degenerate critical points of $ V $.
First, we observe that in the equation $(E_{\l_i})$, the interaction between the bubbles of two different blocks is of the order of  $\e^{(n-2)/2}$ which is negligible in front of the main term whose order is $\e$. This implies that terms of this type will fit into the rest. Second, we note that in the equation $(E_{a_i})$, the interaction between the bubbles of two distinct blocks is of order ($\e^{(n-1)/2}$ if $n\geq 5$; $(\e /|\ln \e |)^{3/2}$ if $n=4$), which is negligible compared to the main term of order ($\e^{2-\frac{2}{n}}$ if $n\geq 5$; $|\ln\e |^{-3/4}\e^{3/2}$ if $n=4$). Consequently, terms of this form can be incorporated into the remaining terms.  Hence, arguing as in the proof of Theorem \ref{th:t1} and taking a new system $((S_1), \cdots , (S_m))$ with each $ (S_i) $ represents the  system studied in the proof of Theorem \ref{th:t1}, the proof  of the theorem follows.
\end{pfn}

\section{Appendix}
In this appendix, we compile several estimates necessary for the analysis presented in the paper. We begin with the following result, which is taken from \cite{Dammak} (see Lemma 2.2).
\begin{lem}\label{lD}
Let $a_i, a_j, \l_i$ and $\l_j$ be such that $\e_{ij}$ is small. For $\alpha$ and $\beta$ satisfying  $ \alpha + \beta = {2n}/(n-2) $ and   $ \alpha \neq \beta $,  it holds 
$$ \int_{\O}\d_{i}^{\alpha}\d_{j}^{\beta}\leq c\e_{ij}^{\min(\alpha,\beta)}.$$
\end{lem}
We now prove some estimations involving bubbles and the approximate solutions. 
\begin{lem}\label{vphi}
 For $ i\in\{1, \cdots ,N\} $ and $j\in\{1, \cdots ,n\}$, let
	$$\varphi_i\in\left\{\pi\d_{a_i, \l_i},\l_{i}\frac{\partial(\pi\d_{a_i, \l_i})}{\partial \l_{i}},\frac{1}{\l_{i}}\frac{\partial(\pi\d_{a_i, \l_i})}{\partial(a_{i})_j}\right\} . $$
Let $\underline{u}:=\sum_{i=1}^N\alpha_{i}\pi\d_{a_i, \l_i}$ and $v\in E_{a,\l}^{\perp}$, it holds 
	$$\left|\int_{\O}(\underline{u})^{p-\e-1}v\varphi_{i}\right|\leq c\|v\|\Big(\e+T_2(\l_i)+\sum_{j\neq i}T_{3}(\e_{ij})\Big)$$
	where $T_2(\l)$ and $ T_{3}(\e_{ij}) $  are defined in Lemma \ref{lfv}. 
\end{lem}
\begin{proof}
 Observe that, from  \eqref{est0}, we deduce that $ \left| \varphi_{i}\right|\leq c\pi \delta_{a_{i},\lambda_{i}}$. Thus, using the fact that: for $ b_1, b_2, z \in \R $ such that $ | z | \leq c | b_1 | $,  
\be \label{lst42}   \big| | b_1 + b_2 |^\gamma z - | b_1 | ^\gamma z \big| \leq c \begin{cases}
 | b_1 b_2 | ^{ (\gamma+1 ) /2 } & \mbox{ if } \gamma \leq 1 , \\
 | b_1 |^\gamma | b_2 | + | b_2 |^\gamma | b_1 | & \mbox{ if } \gamma > 1 , 
 \end{cases} \ee 
 it follows that
	\begin{align*}
 \int_{\Omega}(\underline{u})^{p-\varepsilon-1}v\varphi_{i}= & \alpha_{i}^{p-\varepsilon-1}\int_{\Omega}\left( \pi \delta_{a_{i},\lambda_{i}}\right)^{p-\varepsilon-1}v\varphi_{i}\notag\\
 & + \sum _{j\neq i}\begin{cases}
 O\left( \int_{\Omega} \left( \delta_{a_{i},\lambda_{i}}\delta_{a_{j},\lambda_{j}}\right)^{\frac{p}{2}} \left| v\right|\right)& \text{ if } \; \; n\geq 6 ,\\
 O\left( \int_{\Omega} \delta_{a_{i},\lambda_{i}}^{p-1}\delta_{a_{j},\lambda_{j}}\left| v\right|+\int_{\O}\delta_{a_{i},\lambda_{i}}\delta_{a_{j},\lambda_{j}}^{p-1}\left| v\right| \right) & \text{ if } \; \; n\leq 5.
		\end{cases}
	\end{align*}
	Note that, for $n \leq 5$ (which implies that $p-1>1$), for each $ k $ and $ \ell $, it holds 
$$  \int_{\Omega} \delta_{a_{ k },\lambda_{ k }}^{p-1}\delta_{a_{ \ell },\lambda_{ \ell }} \left| v\right|\leq c\left\| v\right\| \left( \int_{\Omega} \left( \delta_{a_{k},\lambda_{k}}^{p-1}\delta_{a_{ \ell },\lambda_{ \ell }}\right)^{\frac{2n}{n+2}}\right)^{\frac{n+2}{2n}}\leq c\left\| v\right\|\varepsilon_{ k \ell }
$$
by using Lemma \ref{lD}.  However, using \eqref{eq:f1}, we derive that
$$ \int_{\Omega}\left( \delta_{a_{i},\lambda_{i}}\delta_{a_{j},\lambda_{j}}\right)^{\frac{p}{2}} \left| v\right|\leq c \left\| v\right\| \left(\int_{\Omega}\left( \delta_{a_{i},\lambda_{i}}\delta_{a_{j},\lambda_{j}}\right)^{\frac{n}{n-2}} \right)^{\frac{n+2}{2n}} \leq c\left\| v\right\|\varepsilon_{ij}^{\frac{n+2}{2(n-2)}}\left(\ln \varepsilon_{ij}^{-1}\right)^{\frac{n+2}{2n}}.
$$
For the other integral, following the end of the proof of Lemma 6.6 of \cite{2ABE}, we deduce that 
$$
 \left| \int_{\Omega}\left( \pi\delta_{a_{i},\lambda _{i}}\right)^{p-\varepsilon -1}v \varphi _ i \right| \leq c \left\| v\right\| \left( \varepsilon + T_{2}\left(\lambda_{i}\right)\right).
$$
Hence, the proof of the Lemma is completed.
\end{proof}

\begin{lem} \label{dj-thetai} Let $a_1, a_2 \in \O$ be such that $ d(a_i, \partial \O) \geq c >$ and $ \l_1, \l_2$ be large reals such that $ \e_{12} $ is small. Let $ \psi_1 \in \{ \th_{a_1, \l_1} , \l_1 \partial \th_{a_1,\l_1} / \partial \l_1 \}$ . It holds that
$$   \int_\O \d_{a_2, \l_2} ^{\frac{n+2}{n-2} }  | \psi _1 | \leq  \frac{ c }{ ( \l_1 \l_2)^{(n-2) / 2 } } + c \, \e_{12} \underbrace{ \begin{cases}
1 / \l_1 & \mbox{ if } n = 3, \\
\ln \l_1 / \l_1^2 + | a_1 - a_2 | ^2 | \ln | a_1-a_2 | | + 1 / \l_2^{3/2}  & \mbox{ if } n = 4, \\$$
 1 / \l_1^2 + | a_1 - a_2 | ^2 + 1 / \l_2^{3/2}  & \mbox{ if } n \geq 5 .
 \end{cases}  } _{ \Xi_{12} }$$ 
\end{lem}
\begin{proof} First, we remark that we can modify the proof to improve the power of $ \l_2 $. It can be $ \beta $ with  $ \beta < 2 $. \\
Let $ B_1 := B(a_1, r_1 ) $ where $ r_1:= (1/2) \min ( d_{a_1}, e^{-4})$, using \eqref{est1} and \eqref{est2}, we get
$$  \int_\O \d_{a_2, \l_2} ^{\frac{n+2}{n-2} } | \psi _1 | \leq  c \int_{B_1} R^ 1_{a_1, \l_1}( x) \d_{a_2, \l_2} ^{\frac{n+2}{n-2} } \d _{a_1, \l_1} + \int_{ \O \setminus B_1} \d_{a_2, \l_2} ^{\frac{n+2}{n-2} } \d _{a_1, \l_1}   
$$
where $ R^ 1_{a_1, \l_1} $ is defined in  \eqref{est3}. For the last integral, it is easy to get 
$$ \int_{ \O \setminus B_1} \d_{a_2, \l_2} ^{\frac{n+2}{n-2} } \d _{a_1, \l_1}  \leq \frac{c}{\l_1 ^{(n-2) / 2 }} \int_{ \R^n} \d_{a_2, \l_2} ^{\frac{n+2}{n-2} }  \leq \frac{ c } {\l_1 ^{(n-2) / 2 } \l_2 ^{(n-2) / 2 } } . $$
Concerning the other one, using the formula of $ R^1_{a_1, \l_1}$, we have to distinguish three cases.
\begin{itemize}
\item For $ n = 3 $, we remark that, in this case, $ | x-a  | \d_{a, \l} \leq c / \sqrt{\l} $. Thus,  it holds
\begin{align*}
 \int_{B_1} R^ 1_{a_1, \l_1} ( x ) \d_{a_2, \l_2} ^{\frac{n+2}{n-2} } \d _{a_1, \l_1} &  \leq \frac{ c }{ \l_1} \int_{ \R^n}  \d_{a_2, \l_2} ^{\frac{n+2}{n-2} } \d _{a_1, \l_1} + \int_{B_1} | x - a_1 | \d_{a_2, \l_2} ^{\frac{n+2}{n-2} } \d _{a_1, \l_1} \\
 & \leq \frac{c}{ \l_1} \e_{12} + \frac{ c }{ \sqrt{\l_1} } \int_{ \R^n}  \d_{a_2, \l_2} ^{\frac{n+2}{n-2} }  \leq \frac{c}{ \l_1} \e_{12} + \frac{ c }{ \sqrt{\l_1 \l_2} } . 
 \end{align*}
\item For $ n \geq 5 $, it holds that  
\begin{align*}
 \int_{B_1} R^ 1_{a_1, \l_1} ( x ) \d_{a_2, \l_2} ^{\frac{n+2}{n-2} } \d _{a_1, \l_1} &  \leq \frac{ c }{ \l_1^2} \int_{ \R^n}  \d_{a_2, \l_2} ^{\frac{n+2}{n-2} } \d _{a_1, \l_1} + c \int_{B_1} | x - a_1 |^2 \d_{a_2, \l_2} ^{\frac{n+2}{n-2} } \d _{a_1, \l_1} \\
 & \leq c \frac{\e_{12}}{ \l_1^2}  + c  | a_1 - a_2 |^2 \int_{ \R^n}  \d_{a_2, \l_2} ^{\frac{n+2}{n-2} } \d _{a_1, \l_1} +  c \int_{ \O}  | x - a_2 |^2 \d_{a_2, \l_2} ^{\frac{n+2}{n-2} } \d _{a_1, \l_1}  . 
 \end{align*}
Notice that 
\begin{align}
 \int_{ \O}  | x - a_2 |^2 \d_{a_2, \l_2} ^{\frac{n+2}{n-2} } \d _{a_1, \l_1} & =  \int_{ \O} \frac{c}{ \l_2 ^{3/2} | x - a_2 |}   ( \d_{a_2, \l_2} ^{\frac{n-1}{n-2} } \d _{a_1, \l_1}) \notag  \\
 &  \leq \frac{ c } { \l_2 ^{ 3/2} }  \Big( \int_\O \frac{ 1 }{ | x-a_2 |^{ 2 n / 3 } }  \Big)^{3 / (2n) }  \Big( \int_\O \Big[ \d_{a_2, \l_2} ^{\frac{n-1}{n-2} } \d _{a_1, \l_1})\Big]^{\frac{2n}  {2n - 3 } } \Big)^{(2 n - 3 )/(2n)}  \notag \\
  & \leq   \frac{ c } { \l_2 ^{ 3/2} }  \e_{12}  . \label{wx1}
  \end{align}
\item For $ n = 4 $, 
$$
 \int_{B_1} R ^ 1_{a_1, \l_1} ( x ) \d_{a_2, \l_2} ^{\frac{n+2}{n-2} } \d _{a_1, \l_1}   \leq c \frac{ \ln \l_1 }{ \l_1^2} \int_{ \R^n}  \d_{a_2, \l_2} ^{\frac{n+2}{n-2} } \d _{a_1, \l_1} + \int_{B_1} | x - a_1 |^2 | \ln | x - a_1 | | \d_{a_2, \l_2} ^{\frac{n+2}{n-2} } \d _{a_1, \l_1} .
$$
Concerning the second integral, observe that, if $ | a_1 - a_2 | \geq e^{-4} $, it follows that $ | x - a_2| \geq e^{-4}/2 $ for each $ x \in B_1$  and therefore
$$ \int_{B_1} | x - a_1 |^2 | \ln | x - a_1 | | \d_{a_2, \l_2} ^{ 3 } \d _{a_1, \l_1} \leq \frac{c}{ \l_2 ^{ 3 }} \int_{B_1} | x - a_1 |^2 | \ln | x - a_1 |  \d _{a_1, \l_1} . 
$$

In the other case, that is $ | a_1 - a_2 | \leq e^{-4} $, it follows that $  | x - a_2 | \leq (3/2) e^{-4} $, and therefore, since the function : $t \mapsto - t^2 \ln t $ is increasing and convex on the set $(0,e^{-4})$, we deduce that 
$$ | x - a_1 |^2 | \ln | x - a_1 | | \leq c | x - a_2 |^2 | \ln | x - a_2 | | + c | a_2 - a_1 |^2 | \ln | a_2 - a_1 | |. $$
Thus the integral becomes 
\begin{align*} 
& \int_{B_1} | x - a_1 |^2 | \ln | x - a_1 | | \d_{a_2, \l_2} ^{ 3 } \d _{a_1, \l_1} \\
 & \leq c \int_{B_1} | x - a_2 |^2 | \ln | x - a_2 | | \d_{a_2, \l_2} ^{ 3 } \d _{a_1, \l_1} + c  | a_2 - a_1 |^2 | \ln | a_2 - a_1 | | \int_{B_1}  \d_{a_2, \l_2} ^{ 3 } \d _{a_1, \l_1} \\
  & \leq  \int_{ \O} \frac{c | \ln | x-a_2 | | }{ \l_2 ^{3/2} | x - a_2 |}   ( \d_{a_2, \l_2} ^{ 3/2 } \d _{a_1, \l_1}) +  c  | a_2 - a_1 |^2 | \ln | a_2 - a_1 | |  \e_{12} \\
  & \leq \frac{ c } { \l_2 ^{ 3/2} }  \e_{12} + c  | a_2 - a_1 |^2 | \ln | a_2 - a_1 | |  \e_{12} , 
\end{align*}
where we have used the same computations done in \eqref{wx1}.
\end{itemize}
This completes the proof of the lemma.
\end{proof}

\begin{lem}\label{ajout1} Let $ n \geq 4 $. \\
 $(1)$  Let $ d_0 $ be a fixed small positive constant,  $ a  \in \O_0 := \{ x \in \O : d(x , \partial \O) \geq d_0 \} $ and $ \l  $ be large. Then we have
\begin{align}  
 &  \| \pi \d_{a, \l} \| ^2 =  S_n + O\left(\frac{ \ln ^{\s_n} \l  }{\l ^{ 2 }}\right)  \,  , \, \,   \langle \pi \d_{a, \l} ,  \varphi   \rangle  = O\left(\frac{  \ln ^{\s_n} \l   }{\l ^{ m }}\right)  \mbox{ with }   \begin{cases} 
 m = 2 \mbox{ if } \varphi = {\l } {\partial (\pi \d_{a, \l} ) } / {\partial \l} ,   \\
 m= 3 \mbox{ if } \varphi = {\l }^{-1} {\partial (\pi \d_{a, \l} ) } / {\partial a } ,  \end{cases}   \label{tab1}  \\
 &  \Big\|  \l \frac{\partial \pi \d_{a, \l} }{\partial \l } \Big\| ^2 = c + O\left(\frac{ \ln ^{\s_n} \l  }{\l ^{ 2 }}\right)  \quad , \quad  
 \langle  \l \frac{\partial \pi \d_{a, \l} }{\partial \l } ,  \frac{1}{\l }\frac{\partial \pi \d_{a, \l}}{\partial a } \rangle  = O\left(\frac{ \ln ^{\s_n} \l  }{\l ^{ 3 }}\right)  ,   \label{tab2} \\
 &   \Big\| \frac{1}{\l }\frac{\partial \pi \d_{a, \l} }{\partial a ^j} \Big\|^2  = c''  + O\left(\frac{ \ln ^{\s_n} \l  }{\l ^{ 3 }}\right) \quad , \quad 
  \langle \frac{1}{\l }\frac{\partial\d_{a, \l}}{\partial a ^j} ,  \frac{1}{\l }\frac{\partial\d_{a, \l}}{\partial a ^\ell } \rangle =  O\left(\frac{ \ln ^{\s_n} \l  }{\l ^{ 3 }}\right)  \, \, \forall \ell \neq j , \label{tab3} 
   \end{align}
 $(2)$  Let  $ a_1 , a_2 \in \O_0 $ and $ \l_1 , \l_2 $ be large so that $ \e_{12} $, defined in \eqref{eq:3}, is small. Then we have 
\be \label{tab4}  \langle \varphi_1 , \varphi _2 \rangle = O ( \e _ {12} ) \quad \forall \varphi_i \in \Big\{ \pi \d_{a_i, \l_i} ,  \l_i \frac{\partial \pi \d_{a_i, \l_i} }{\partial \l _i } , \frac{1}{ \l _i}  \frac{\partial \pi \d_{a _i, \l_i} }{\partial a_i^j } \Big\}  , \, \, \mbox{ for } \, \, i \in \{1, 2\} \mbox{ and } j \in \{ 1, \cdots , n \} . \ee
\end{lem}
\begin{pf} Note that, the first assertion is proved in \eqref{eqq:e1}, \eqref{eqq:e2} and \eqref{eqq:e3} respectively. Concerning  Claim \eqref{tab2}, using \eqref{est0} and \eqref{est1}, we get 
\begin{align*}
\Big\|  \l \frac{\partial \pi \d_{a, \l} }{\partial \l } \Big\| ^2 & = \int_\O  \Big[( - \D + V ) \l \frac{\partial \pi \d_{a, \l} }{\partial \l } \Big] \l \frac{\partial \pi \d_{a, \l} }{\partial \l } = \frac{n+2}{n-2} \int _\O \d _{ a, \l} ^{4/(n-2) } \l \frac{\partial  \d_{a, \l} }{\partial \l } \Big( \l \frac{\partial \d_{a, \l} }{\partial \l } - \l \frac{\partial \theta_{a, \l} }{\partial \l } \Big) \\
 & = \frac{n+2}{n-2}   \int _{\R^n} \d _{ a, \l} ^{4/(n-2) } \Big(\l \frac{\partial  \d_{a, \l} }{\partial \l } \Big)^2 + O \Big( \int_{ \R^n \setminus \O}  \d _{ a, \l} ^{2n /(n-2) }  + \int_\O  \d _{ a, \l} ^{(n+2)/(n-2) }   \theta _{ a, \l} \Big) .
\end{align*}
Thus, \eqref{eq:e3} and easy computations imply the proof of the first equality in \eqref{tab2}. For the other equality, observe that 
$$  \langle  \l \frac{\partial \pi \d_{a, \l} }{\partial \l } ,  \frac{1}{\l }\frac{\partial \pi \d_{a, \l}}{\partial a } \rangle = \int_\O \Big[ ( - \D + V ) \l \frac{\partial \pi \d_{a, \l} }{\partial \l } \Big] \frac{1}{ \l} \frac{\partial \pi \d_{a, \l} }{\partial a } = \frac{n+2}{n-2} \int _\O \d _{ a, \l} ^{\frac{4}{n-2} } \l \frac{\partial  \d_{a, \l} }{\partial \l } \Big( \frac{1}{\l } \frac{\partial \d_{a, \l} }{\partial a } - \frac{1}{\l } \frac{\partial \theta_{a, \l} }{\partial a } \Big) . $$
Let $ B:= B(a, d_0 )$, by oddness, it holds that
$$ \int _\O \d _{ a, \l} ^{\frac{4}{n-2} } \l \frac{\partial  \d_{a, \l} }{\partial \l }  \frac{1}{\l } \frac{\partial \d_{a, \l} }{\partial a } = \int _{\O\setminus B } \d _{ a, \l} ^{\frac{4}{n-2} } \l \frac{\partial  \d_{a, \l} }{\partial \l }  \frac{1}{\l } \frac{\partial \d_{a, \l} }{\partial a } = O \Big( \frac{1}{ \l^{n+1}} \Big) . $$
For the other integral, since $ a \in \O_0$,  we get
\begin{align*}
 \frac{n+2}{n-2}  \int _\O \d _{ a, \l} ^{\frac{4}{n-2} } \l \frac{\partial  \d_{a, \l} }{\partial \l }  \frac{1}{\l } \frac{\partial \theta_{a, \l} }{\partial a } & = \int_\O \Big[ ( - \D + V ) \l \frac{\partial \pi \d_{a, \l} }{\partial \l } \Big] \frac{1}{ \l} \frac{\partial  \theta_{a, \l} }{\partial a } \\
 & = \int_\O  \l \frac{\partial \pi \d_{a, \l} }{\partial \l }   \Big[ ( - \D + V )\frac{1}{ \l} \frac{\partial  \theta_{a, \l} }{\partial a } \Big]  - \int_{ \partial \O} \frac{\partial }{ \partial \nu} \Big[  \l \frac{\partial \pi \d_{a, \l} }{\partial \l } \Big] \frac{1}{ \l} \frac{\partial  \theta_{a, \l} }{\partial a }  \\
  & =  \int_\O  \l \frac{\partial \pi \d_{a, \l} }{\partial \l }  V \frac{1}{ \l} \frac{\partial  \d_{a, \l} }{\partial a } + O \Big( \frac{1}{ \l^{ n-1} } \Big) . 
 \end{align*}
To complete the proof,  using \eqref{est0}, \eqref{est1}, \eqref{est2} and the fact that $ | \partial \d_{a, \l} / \partial a | \leq c \d_{a, \l} / | x-a | $, we obtain 
\begin{align*}
&   \Big | \int_{ \O \setminus B  }  \l \frac{\partial \pi \d_{a, \l} }{\partial \l }  V \frac{1}{ \l} \frac{\partial  \d_{a, \l} }{\partial a }  \Big |  \leq c \int_{ \O \setminus B  } \d_{a, \l}  \Big  | \frac{1}{ \l} \frac{\partial  \d_{a, \l} }{\partial a } \Big | \leq \frac{ c }{ \l^{n-1} }, \\
&  \int_{  B  }  \l \frac{\partial  \pi \d_{a, \l} }{\partial \l }  V \frac{1}{ \l} \frac{\partial  \d_{a, \l} }{\partial a }   = V(a) \int_{  B(a, d_0) } \frac{\partial  \d_{a, \l} }{\partial \l }  \frac{\partial  \d_{a, \l} }{\partial a } + O \Big( \int_{  B } \frac{1}{\l } \d_{a, \l} ^2 + 
\int_{  B } R^1_{a,\l}(x)  \frac{ \d_{a, \l}^2 }{ \l | x - a | }    \Big )    \leq \, c \,  \frac{ \ln ^{ \s_n} \l  }{ \l^{ 3 } } . 
\end{align*}
This completes the proof of Assertion \eqref{tab2}. 

The proof of \eqref{tab3} can be done in the same way than the proof of \eqref{tab2}. Thus, we will omit it here. 

Concerning Assertion \eqref{tab4}, note that, \eqref{est0} and easy computations imply that 
$$ | ( - \D + V ) \varphi _i | \leq c \d_{a_i , \l_i }^{(n+2)/(n-2)}  \qquad \mbox{ and } \qquad  | \varphi_i | \leq c \d_{a_i , \l_i }  \qquad \mbox{ for } \, i \in \{ 1, 2 \}.$$
Thus, using the last inequality of \eqref{dij}, we deduce that 
$$ |  \langle \varphi_1 , \varphi _2 \rangle | = | \int_\O (  - \D + V ) \varphi _1 ) \varphi_2 | \leq c \int_\O \d_{a_1 , \l_1 } ^{ (n+2)/(n-2) } \d _ {a_2 , \l_2 } \leq c \e_{12} .$$
This completes the proof of Assertion \eqref{tab4}. 
\end{pf}

\section{Conclusion}
In this paper, we have investigated the existence of solutions to a nonlinear elliptic problem with Dirichlet boundary conditions, involving a slightly subcritical exponent for Sobolev embedding $H^1_0(\Omega) \hookrightarrow L^q(\Omega)$. Through a careful asymptotic analysis of the gradient of the associated Euler-Lagrange functional near the "bubbles," we successfully constructed solutions that exhibit bubble formation clustered at interior points. The methodology adopted is tailored to variational problems. While our work demonstrates the existence of solutions with clustered bubbles for the given problem, several promising avenues for further research and open questions remain:
\begin{itemize}
 \item[\bf{(i)}] {\bf Location of the Concentration Points:}
 This paper focuses on the construction of interior bubbling solutions with clustered bubbles. A natural extension of this work would be to investigate the existence of solutions that concentrate  at non-isolated interior points that approach the boundary in the limit.
 \item[\bf{(ii)}]  {\bf Impact of the Nature of Critical Points:} 
 The solutions constructed in this paper are based on the assumption that the chosen critical point $b$ of the potential $V$ is non-degenerate, and that the corresponding function $\mathcal{F}_{b,N}$ has a non-degenerate critical point, where $\mathcal{F}_{b,N}$ is defined by \eqref{FYN} with $N$ denotes the number of bubbles clustered at point b. A natural question arises: what occurs if $b$ is degenerate, or if $\mathcal{F}_{b,N}$ lacks non-degenerate critical points?
 \item[\bf {(iii)}]
 {\bf Impact of Subcritical Exponent:} The present work addresses a slightly subcritical exponent for Sobolev embedding. Future investigations could extend the analysis to exponents that are slightly supercritical, that is, when $\varepsilon < 0$ but close to zero.
 \end{itemize}

\bigskip

\noindent
{\bf Acknowledgment.} The Researchers would like to thank the Deanship of Graduate Studies and Scientific Research at Qassim University for financial support (QU-APC-2025).
\bigskip

\noindent {\bf Funding information:} The Article Processing Charges  of this article was supported by the Deanship of Graduate Studies and Scientific Research at Qassim University.
\bigskip

\noindent {\bf Author contributions:} All authors contributed significantly and equally to writing this article. All authors read and approved the final manuscript.
\bigskip

\noindent {\bf Conflict of interest:} The authors state no conflict of interest.
\bigskip

\noindent {\bf Data availability statement:} Data are contained within the article.

\end{document}